\documentclass[twoside,11pt,reqno]{amsart}
\usepackage{amsmath,amssymb,amscd,mathrsfs,stmaryrd,wasysym}

\makeatletter

\newtheorem{Theorem}{Theorem}[section]
\newtheorem{Corollary}[Theorem]{Corollary}

\newtheorem{Conjecture}[Theorem]{Conjecture}
\newtheorem{Lemma}[Theorem]{Lemma}
\theoremstyle{remark}

\numberwithin{equation}{section}

\newbox\squ  
\setbox\squ=\hbox{\vrule width.6pt
   \vbox{\hrule height.6pt width.4em\kern1ex
         \hrule height.6pt}%
   \vrule width.6pt}

\def\nslash{\:\notslash\:}

\def\RtoH{\rho}
\def\HtoR{\sigma}

\def\id{\operatorname{id}}

\def\C{{\mathbb C}}
\def\ell{l}
\def\Q{{\mathbb Q}}
\def\Z{{\mathbb Z}}

\def\0{{\bar 0}}
\def\1{{\bar 1}}

\def\Pol{P\!ol}

\def\B{{\mathcal B}}

\def\End{{\operatorname{End}}}

\def\ch{{\operatorname{ch}\:}}

\def\height{{\operatorname{ht}}}

\def\bi{\text{\boldmath$i$}}
\def\bj{\text{\boldmath$j$}}

\def\eps{{\varepsilon}}
\def\phi{{\varphi}}

\def\Ga{{\Gamma}}
\def\la{{\lambda}}
\def\La{{\Lambda}}
\def\de{{\delta}}

\def\al{{\alpha}}
\def\be{{\beta}}
\def\Tab{\mathscr{T}}

\def\T{{\mathtt T}}
\def\Stab{{\mathtt S}}

\begin{document}

\title[Cyclotomic Hecke algebras]{
Blocks of cyclotomic Hecke algebras and 
khovanov-Lauda algebras}
\author{Jonathan Brundan and Alexander Kleshchev}

\begin{abstract}
We construct an explicit isomorphism between
blocks of cyclotomic Hecke algebras
and (sign-modified) cyclotomic Khovanov-Lauda algebras in type A.
These isomorphisms connect  
the categorification conjecture of Khovanov and Lauda
to Ariki's categorification 
theorem. 
The Khovanov-Lauda algebras are naturally graded, which allows us to 
exhibit a non-trivial 
$\Z$-grading on 
blocks of cyclotomic Hecke algebras, including 
symmetric groups in positive characteristic.
\end{abstract}
\thanks{Supported in part by NSF grant  
DMS-0654147. This work was started at MSRI}
\address{Department of Mathematics, University of Oregon, Eugene, Oregon, USA.}
\email{brundan@uoregon.edu, klesh@uoregon.edu}
\maketitle

\section{Introduction}\label{SIntro}

In \cite{KL1,KL2}, Khovanov and Lauda have introduced a remarkable new 
family of algebras 
and formulated a categorification conjecture 
predicting a tight connection between the
representation theory of these algebras and Lusztig's geometric construction 
of canonical bases \cite{Lubook}.
This paper arose as a first attempt to understand the cyclotomic 
Khovanov-Lauda 
algebras in type A by relating them to cyclotomic Hecke algebras
and their rational degenerations. 

Let $F$ be a fixed ground field
and $q \in F^\times$.
Let $e$ be the smallest positive integer such that
$1+q+\cdots+q^{e-1} = 0$, setting $e := 0$
if no such integer exists.
The main result of the article gives an explicit isomorphism
between blocks of cyclotomic Hecke algebras
associated to the complex reflection groups of type
$G(\ell,1,d)$
if $q \neq 1$, or the corresponding degenerate cyclotomic Hecke algebras 
if $q = 1$, and a sign-modified version of
cyclotomic Khovanov-Lauda algebras of type $\text{A}_\infty$
if $e=0$ 
or type $\text{A}_{e-1}^{(1)}$ if $e > 0$.

For $F$ of characteristic zero, this isomorphism connects  
the categorification conjecture of Khovanov and Lauda 
\cite[$\S$3.4]{KL1} in type A
to Ariki's categorification theorem \cite{Ariki}
and
its degenerate analogue \cite{BKcat}.
It doesn't immediately {\em prove} the Khovanov-Lauda conjecture in 
any of these
cases,
because the conjecture takes into account a natural $\Z$-grading
on cyclotomic Khovanov-Lauda algebras which is hard to identify on 
Hecke algebras.

To formulate the main results precisely, let $\Ga$ be the quiver with vertex set  $I:=\Z/e\Z$,
and a directed edge from $i$ to $j$ if $j  = i+1$.
Thus $\Gamma$ is the quiver of type $\text{A}_\infty$ if $e=0$
or $\text{A}_{e-1}^{(1)}$ if $e > 0$, with a specific orientation:
\begin{align*}
\text{A}_\infty&:\qquad\cdots \longrightarrow-2\longrightarrow -1 \longrightarrow 0 \longrightarrow 1 \longrightarrow 
2\longrightarrow \cdots\\
\text{A}_{e-1}^{(1)}&:\qquad0\rightleftarrows 1
\qquad
\begin{array}{l}
\\
\,\nearrow\:\:\:\searrow\\
\!\!2\,\longleftarrow\, 1
\end{array}
\qquad
\begin{array}{rcl}\\
0&\!\rightarrow\!&1\\
\uparrow&&\downarrow\\
3&\!\leftarrow\!&2
\end{array}
\qquad
\begin{array}{l}
\\
\:\nearrow\quad\searrow\\
\!4\qquad\quad \!1\\
\nnwarrow\quad\quad\,\sswarrow\\
\:\:3\leftarrow 2
\begin{picture}(0,0)
\put(-152.5,41){\makebox(0,0){0}}
\put(-13.5,53.5){\makebox(0,0){0}}
\end{picture}
\end{array}
\qquad \cdots
\end{align*}
The corresponding (symmetric) Cartan matrix
$(a_{i,j})_{i, j \in I}$ is defined by
$$
a_{i,j} := \left\{
\begin{array}{rl}
2&\text{if $i=j$},\\
0&\text{if $i \nslash j$},\\
-1&\text{if $i \rightarrow j$ or $i \leftarrow j$},\\
-2&\text{if $i \rightleftarrows j$}.
\end{array}\right.
$$
Here the symbols 
$i \rightarrow j$ and $j \leftarrow i$
both indicate that $j=i+1\neq i-1$,
$i \rightleftarrows j$ indicates that $j = i+1= i-1$,
and
$i \nslash j$
indicates that $j \neq i, i \pm 1$.

To the index set $I$,
we associate two lattices
\begin{equation}\label{pd}
P := 
\bigoplus_{i \in I} \Z\La_i, \qquad Q :=
\bigoplus_{i \in I} \Z\al_i,
\end{equation}
and let
$(.,.):P \times Q \rightarrow \Z$
be the bilinear pairing
defined by $(\La_i, \alpha_j) := \delta_{i,j}$.
Let $P_+$ (resp.\ $Q_+$) denote the subset of $P$ (resp.\ $Q$)
consisting of the elements
that have non-negative coefficients
when written in terms of the given basis.
For $\al\in Q_+$ and $\La\in P_+$ 
define the {\em height} of $\al$ and the {\em level} of $\La$ as follows:
$$
\height(\al):=\sum_{i\in I}(\La_i,\al),\qquad \ell(\La):=\sum_{i\in I}(\La,\al_i).
$$
Let $S_d$ be the symmetric group with 
basic transpositions $s_1,\dots,s_{d-1}$. 
It acts on the left
on the set 
of $d$-tuples $\bi = (i_1,\dots, i_d) \in I^d$ by place
permutation. 
The $S_d$-orbits on $I^d$ are the sets
\begin{equation*}
I^\alpha := \{\bi=(i_1,\dots,i_d) \in I^d\:|\:\alpha_{i_1}+\cdots+\alpha_{i_d} = \alpha\}
\end{equation*}
parametrized by all $\alpha \in Q_+$ of height $d$. 
 
Let $H_d$ denote
the affine Hecke algebra associated to $S_d$
if $q \neq 1$, or its rational
degeneration if $q=1$.
Thus,
$H_d$ is the $F$-algebra defined by generators
$T_1,\dots,T_{d-1}, X_1^{\pm 1}, \dots, X_d^{\pm 1}$
and relations (\ref{QPoly})--(\ref{QCoxeter})
if $q \neq 1$, 
or by generators
$s_1,\dots,s_{d-1}, x_1,\dots,x_d$ and
relations (\ref{EPoly})--(\ref{ECoxeter2})
if $q=1$.
From now on, fix $\La\in P_+$ of level $\ell$
and let $H^\La_d$ be the corresponding cyclotomic quotient of
$H_d$. Thus,
\begin{equation}\label{ECHA}
H_d^\La := \left\{
\begin{array}{ll}
H_d \Big/ \big\langle \,\textstyle\prod_{i\in I}(X_1-q^i)^{(\La,\al_i)}\,\big\rangle
&\text{if $q \neq 1$,}\\
H_d \Big/ \big\langle \,\textstyle\prod_{i\in I}(x_1-i)^{(\La,\al_i)}\,\big\rangle
&\text{if $q = 1$}.
\end{array}\right.
\end{equation}
We refer to this algebra simply as 
the {\em   cyclotomic Hecke algebra} if $q \neq 1$
and the {\em degenerate cyclotomic Hecke algebra} if $q=1$.

There is a natural system 
$\{e(\bi)\:|\:\bi \in I^d\}$
of mutually orthogonal idempotents in $H_d^\La$; 
see $\S$\ref{QSId} or $\S$\ref{SId}
for more details in the two cases.
In addition to $\La \in P_+$, fix also 
$\al\in Q_+$ of height $d$. 
Let 
\begin{equation}\label{bal}
e_{\alpha} := \sum_{\bi \in I^\alpha} e(\bi) \in H^\La_d.
\end{equation}
As a consequence of \cite{LM} or \cite[Theorem~1]{cyclo},
$e_{\alpha}$ is either zero or it is a primitive central idempotent
in $H^\La_d$.
Hence the algebra
\begin{equation}\label{fe}
H^\La_\alpha := e_{\alpha} H^\La_d
\end{equation}
is either zero or it is a single {\em block}
of the algebra $H^\La_d$.
For $h\in H_d^\La$, 
we still write $h$ for  the projection $e_{\al}h\in H_\al^\La$. 
So 
$H^\La_\alpha$ again has 
generators $T_1,\dots,T_{d-1}$, $X_1^{\pm 1},\dots,X_d^{\pm 1}$ 
if $q \neq 1$,
or
$s_1,\dots,s_{d-1}, x_1,\dots,x_d$
if $q=1$. 
However, it seems difficult to describe a complete set of 
relations between these generators at the level of $H^\La_\alpha$. These 
generators are also not well-suited for defining interesting gradings. So, inspired by \cite{KL1,KL2}, 
we introduce an explicit new set of generators
\begin{equation}\label{EKLGens}
\{e(\bi)\:|\: \bi\in I^\al\}\cup\{y_1,\dots,y_{d}\}\cup\{\psi_1, \dots,\psi_{d-1}\}
\end{equation}
of $H_\al^\La$, which we call the {\em Khovanov-Lauda generators}; see
(\ref{QEPolKL}) and (\ref{QQCoxKL}) or
(\ref{jon}) and (\ref{ECoxKL})
in the two cases.

\vspace{2mm}
\noindent
{\bf Main Theorem.} {\em
The algebra $H_\al^\La$ is generated by the elements 
(\ref{EKLGens}) subject 
only to the following relations for $\bi,\bj\in I^\al$ and all 
admissible $r, s$:
\begin{align}
y_1^{(\La,\al_{i_1})}e(\bi)&=0;\label{ERCyc}\\
e(\bi) e(\bj) &= \de_{\bi,\bj} e(\bi);
\hspace{11.3mm}{\textstyle\sum_{\bi \in I^\alpha}} e(\bi) = 1;\label{R1}\\
y_r e(\bi) &= e(\bi) y_r;
\hspace{20mm}\psi_r e(\bi) = e(s_r{\cdot}\bi) \psi_r;\label{R2PsiE}\\
\label{R3Y}
y_r y_s &= y_s y_r;\\
\label{R3YPsi}
\psi_r y_s  &= y_s \psi_r\hspace{42.4mm}\text{if $s \neq r,r+1$};\\
\psi_r \psi_s &= \psi_s \psi_r\hspace{41.8mm}\text{if $|r-s|>1$};\label{R3Psi}\\
\psi_r y_{r+1} e(\bi) 
&= 
\left\{
\begin{array}{ll}
(y_r\psi_r+1)e(\bi) &\hbox{if $i_r=i_{r+1}$},\\
y_r\psi_r e(\bi) \hspace{28mm}&\hbox{if $i_r\neq i_{r+1}$};
\end{array}
\right.
\label{R6}\\
y_{r+1} \psi_re(\bi) &=
\left\{
\begin{array}{ll}
(\psi_r y_r+1) e(\bi) 
&\hbox{if $i_r=i_{r+1}$},\\
\psi_r y_r e(\bi)  \hspace{28mm}&\hbox{if $i_r\neq i_{r+1}$};
\end{array}
\right.
\label{R5}\\
\psi_r^2e(\bi) &= 
\left\{
\begin{array}{ll}
0&\text{if $i_r = i_{r+1}$},\\
e(\bi)&\text{if $i_r \nslash i_{r+1}$},\\
(y_{r+1}-y_r)e(\bi)&\text{if $i_r \rightarrow i_{r+1}$},\\
(y_r - y_{r+1})e(\bi)&\text{if $i_r \leftarrow i_{r+1}$},\\
(y_{r+1} - y_{r})(y_{r}-y_{r+1}) e(\bi)\!\!\!&\text{if $i_r \rightleftarrows i_{r+1}$};
\end{array}
\right.
 \label{R4}\\
\psi_{r}\psi_{r+1} \psi_{r} e(\bi)
&=
\left\{\begin{array}{ll}
(\psi_{r+1} \psi_{r} \psi_{r+1} +1)e(\bi)&\text{if $i_{r+2}=i_r \rightarrow i_{r+1}$},\\
(\psi_{r+1} \psi_{r} \psi_{r+1} -1)e(\bi)&\text{if $i_{r+2}=i_r \leftarrow i_{r+1}$},\\
\big(\psi_{r+1} \psi_{r} \psi_{r+1} -2y_{r+1}
\\\qquad\:\quad +y_r+y_{r+2}\big)e(\bi)
\hspace{2.4mm}&\text{if $i_{r+2}=i_r \rightleftarrows i_{r+1}$},\\
\psi_{r+1} \psi_{r} \psi_{r+1} e(\bi)&\text{otherwise}.
\end{array}\right.
\label{R7}
\end{align}
}

A striking feature of the above theorem is that
the relations depend only on the quiver $\Ga$
(hence $e$) but 
do not involve the parameter $q$.

\vspace{2mm} 

\noindent
{\bf Corollary 1.}
{\em There is a unique $\Z$-grading on
$H^\La_\alpha$ such that
$e(\bi)$ is of degree 0,
$y_r$ is of degree $2$, and
$\psi_r e(\bi)$ is of degree $-a_{i_r,i_{r+1}}$
for each $r$ and $\bi \in I^\alpha$.}

\vspace{2mm}

\noindent
{\bf Corollary 2.}
{\em
Suppose that $F$ is of characteristic zero.
Then the algebra
$H^\La_d$ for $q$ not a root of unity
is isomorphic to the
algebra $H^\La_d$ for $q=1$.
In other words, the
  cyclotomic Hecke algebra for generic $q$
is isomorphic to its rational degeneration.
}

\vspace{2mm}

The presentation for the algebra $H^\La_\alpha$ in our Main Theorem
is a sign-modified version of the presentation for the cyclotomic
Khovanov-Lauda algebra associated to the quiver $\Ga$ and 
the weight $\La$ as defined in \cite[$\S$3.4]{KL1};
this sign-modified version was introduced already in \cite{KL2}
(except in the case $e=2$).
Thus the theorem shows that
blocks of cyclotomic Hecke algebras
are cyclotomic Khovanov-Lauda algebras.

If $\La$ is of level one, $H_d^\La$ is isomorphic to
the 
group algebra $F S_d$ of the symmetric group if $q=1$ or
the associated Iwahori-Hecke algebra for arbitrary $q$.
In these cases, 
our Main Theorem for $e=0$
is essentially Young's semi-normal
form (see $\S$5), while for $e > 0$
Corollary 1 yields interesting 
$\Z$-gradings on
blocks of symmetric groups and the associated Iwahori-Hecke algebras.
The existence of such gradings was predicted already 
by Rouquier \cite[Remark 3.11]{R1} and Turner \cite{T}.
This means that it is now possible to study
{\em graded} representation theory of these algebras.
First steps in this direction are taken in \cite{BKW},
in which we construct graded versions of Specht modules.

Corollary 2 can be viewed as an extension of Lusztig's
results from \cite{Lu} in type A.
As an application, it is easy to see 
that Ariki's categorification theorem from
\cite{Ariki} when $q$ is not a root of unity
is simply equivalent to the degenerate analogue proved (in a quite
different way) in
 \cite{BKcat}.
Our Main Theorem also makes possible the comparison of
blocks of   cyclotomic Hecke algebras
at a primitive complex $p$th root of unity 
with corresponding blocks of degenerate cyclotomic 
Hecke algebras over fields of 
characteristic $p$; see $\S$6 for further discussion.

Before we started work on this article (and before \cite{KL1}
became available),
the first author jointly with Stroppel
made calculations as part of \cite{BS3}
that are equivalent to the Main Theorem for $\La$ of level two,
$q=1$ and $e=0$.
In that case, the algebras $H^\La_\alpha$
are Morita equivalent 
to (slightly generalised versions of)
Khovanov's diagram algebra.
This connection is exploited further
in \cite{BS3} to give an elementary proof of the Khovanov-Lauda
categorification conjecture for level two weights in 
type $\text{A}_\infty$.

Since completing this work, we have learnt from
Rouquier that he has independently discovered essentially the same 
family of algebras as Khovanov and Lauda at the affine level;
he refers to them as
{\em quiver nil-Hecke algebras}.
In particular, in \cite[$\S$3.2.6]{R2}, Rouquier
has proved an analogue of our Main Theorem for the
affine algebras (suitably localized).

The rest of the article is taken up with the proof of the Main Theorem.
The strategy is clear: let $R^\La_\alpha$ be the algebra defined by generators
and relations as in the Main Theorem. Then we need to 
construct mutually inverse
homomorphisms 
$\RtoH:R^\La_\alpha \rightarrow H^\La_\alpha$ and
$\HtoR:H^\La_\alpha \rightarrow R^\La_\alpha$
by explicitly checking relations in both directions.
However, there are many subtle differences in the two cases 
$q \neq 1$ and $q=1$, 
so we carry out the two sets of calculations independently, 
treating the easier degenerate case in
$\S$3 then the non-degenerate case in $\S$4.
In $\S$2 we make a few definitions that are common to both cases.
Finally in $\S$5 and $\S$6 we discuss briefly the relation to 
Young's semi-normal form and make some observations about base change.

\vspace{2mm}
\noindent
{\em Acknowledgements.} We thank Anton Cox for pointing out a gap
in an earlier version, and Meinolf Geck for drawing our 
attention to the fact that our results could be used to prove
a conjecture of Mathas.

\section{Preliminaries}

\subsection{Divided difference operators}\label{SSBN}
The symmetric group $S_d$ 
acts on the left 
on the polynomial ring $F[y_1,\dots,y_d]$ and on the ring of power series 
$F[[y_1,\dots,y_d]]$ by permuting variables; we often
denote $w{\cdot}f$ by ${^{w\!}}f$ for $w\in S_d$ and $f\in F[[y_1,\dots,y_d]]$. 
For any $f\in F[y_1,\dots,y_d]$
and $1 \leq r < d$, 
the {\em divided difference} operator $\partial_r$ is defined by
\begin{equation}\label{EPartial}
\partial _r(f):=\frac{{^{s_r\!}}f -f}{y_r-y_{r+1}}
\end{equation}
It extends to $F[[y_1,\dots,y_d]]$ by continuity. 
We will need the {\em product rule} for
divided difference operators: for $f,g\in F[[y_1,\dots,y_d]]$ we have that
\begin{align}\label{EPR}
\partial_r(fg)&=\partial_r(f)g+{^{s_r\!}}f\partial_r(g) 
= \partial_r(f) {^{s_r\!}} g + f \partial_r(g).
\end{align}
Note as a matter of notation here that ${^{s_r\!}} fg$ means
$({^{s_r\!}}f) g$ not ${^{s_r\!}}(fg)$.

We will often be given an $F$-algebra $A$
and commuting nilpotent elements $y_1,\dots,y_d \in A$.
There is then an algebra homomorphism
\begin{equation}\label{conv}
F[[y_1,\dots,y_d]] \rightarrow A
\end{equation}
mapping each $y_r$ to the element of $A$ with the same name.
Given also power series $f, g \in F[[y_1,\dots,y_d]]$,
this homomorphism allows us to interpret
expressions like
${^{s_r\!}}f$, $\partial_r(g)$
and $f / g$
as elements of $A$; we mean the element of $A$ obtained
by first evaluating the given expression as a power series 
in $F[[y_1,\dots,y_d]]$
and only then taking the image of the result under the
homomorphism (\ref{conv}).
We stress that $f$ and $g$ must be given specifically
as power series (not merely as elements of $A$)
for such expressions to make sense.

\subsection{Cyclotomic Khovanov-Lauda algebras}\label{sckl}
For the remainder of the article, we fix notation exactly
as in the statement of the Main Theorem.
In particular, $\Ga$ is the quiver with vertex set $I = \Z / e\Z$
defined at the beginning of the introduction and $P, Q$ are 
as in (\ref{pd}).
The (sign-modified) {\em cyclotomic
Khovanov-Lauda algebra} of type $\Ga$
associated to $\La \in P_+$ of level $\ell$
and $\alpha \in Q_+$ of height $d$
is the
$F$-algebra $R^\La_\alpha$  defined by the generators
\begin{equation}
\label{kl0}
\{e(\bi)\:|\: 
\bi\in I^\al\}\cup\{y_1,\dots,y_{d}\}\cup\{\psi_1, \dots,\psi_{d-1}\}
\end{equation}
subject to the relations (\ref{ERCyc})--(\ref{R7}) above. 

In view of the first lemma below,
we are
in the situation of (\ref{conv}), so have a homomorphism
\begin{equation}\label{Rhom}
F[[y_1,\dots,y_d]] \rightarrow R^\La_\alpha
\end{equation} 
mapping $y_1,\dots,y_d$
to the elements of $R^\La_\alpha$ with the same name.
By (\ref{R3YPsi}), (\ref{R6}) and (\ref{R5}),
the following useful identity holds in $R^\La_\alpha$:
\begin{equation}
\label{EDDR}
f\psi_r e(\bi)=
\left\{\begin{array}{ll}
\psi_r{^{s_r\!}} fe(\bi)+\partial_r(f)e(\bi)&\text{if $i_r = i_{r+1}$},\\
\psi_r{^{s_r\!}} fe(\bi)&\text{otherwise},
\end{array}\right.
\end{equation}
for any $1\leq r<d$, $\bi \in I^\alpha$ and $f\in F[[y_1,\dots,y_d]]$.

\begin{Lemma}\label{LNilY}
The elements $y_r\in R_\al^\La$ are nilpotent for all $1\leq r\leq d$. 
\end{Lemma}

\begin{proof}
It suffices to prove that each $y_re(\bi)$ is nilpotent. Apply 
induction on $r$. 
The base case $r=1$ follows as $(y_1e(\bi))^{(\La,\al_{i_1})}=0$ 
in $R_\al^\La$. For the induction step, 
we assume that 
$y_r$ is nilpotent for some $1 \leq r < d$
and prove that
$y_{r+1}e(\bi)$ is too. 
We consider two cases.

If $i_r=i_{r+1}$, set  
$\tau_r:=\psi_r(y_r-y_{r+1})+1$.
One checks that  
$\tau_r^2e(\bi)=e(\bi)$ and 
$\tau_ry_r\tau_re(\bi)=y_{r+1}e(\bi)$.
Hence $y_{r+1}^n e(\bi) =
\tau_r y_r^n \tau_r e(\bi)$,
and the nilpotency of $y_{r+1}e(\bi)$
follows from that of $y_r$.

If $i_r \neq i_{r+1}$, then
multiplying the equation
$y_r \psi_r e(\bi) = \psi_r y_{r+1} e(\bi)$ on the left
by $\psi_r$ and using (\ref{R4}) gives that
$\psi_r y_r \psi_r e(\bi) = \pm (y_r - y_{r+1})^k y_{r+1} e(\bi)$
for some $k \in \{0,1,2\}$ and some choice of sign.
Hence $$
y_{r+1}^{k+1} e(\bi) = y_r f e(\bi) \pm \psi_r y_r \psi_r e(\bi)
$$
for some (possibly zero) $f \in F[y_r,y_{r+1}]$ and some sign.
As $y_r$ is nilpotent, we deduce using
(\ref{R4}) that $\psi_r y_r \psi_r e(\bi)$ is nilpotent too,
and of course $y_r f$ is nilpotent as well.
Observing finally that
$y_r f e(\bi)$ and $\psi_r y_r \psi_r e(\bi)$ commute,
this implies the nilpotency of $y_{r+1}^{k+1} e(\bi)$.
Hence $y_{r+1} e(\bi)$ is nilpotent too.
\end{proof}

\begin{Corollary}\label{fda} $R^\La_\alpha$ is a finite dimensional algebra.
\end{Corollary}

\begin{proof}
For each $w \in S_d$, fix a reduced expression
$w = s_{i_1} \cdots s_{i_n}$ and then define
$\psi_w := \psi_{i_1} \cdots \psi_{i_n}$.
By an easy application of the relations
(\ref{R1})--(\ref{R7}) one checks as in \cite[$\S$2.3]{KL1}
that $R^\La_\alpha$
is spanned by the elements
\begin{equation}\label{sset}
\{\psi_w y_1^{n_1} \cdots y_d^{n_d} e(\bi)\:|\:
w \in S_d, \bi \in I^\alpha, n_1,\dots,n_d \geq 0\}.
\end{equation}
It remains to observe by Lemma~\ref{LNilY} that all but finitely many of these elements
are zero.
\end{proof}

The following conjecture is a quite trivial 
consequence of the relations in the case $\ell = 1$, and is 
established
for $\ell = 2$ in \cite{BS3}.
We couldn't see how to check even the $\ell = 2$
case using the relations alone.

\begin{Conjecture}
If $e=0$ and $\La$ is of level $\ell$ 
then $y_r^\ell=0$ in $R_\al^\La$ for any $1\leq r\leq d$. 
\end{Conjecture}

\section{The degenerate case}

\subsection{\boldmath Blocks of degenerate cyclotomic Hecke algebras}\label{SId}
For $d \geq 0$, let $H_d$ be the degenerate 
affine Hecke algebra, working always over the fixed field
$F$ of characteristic $e \geq 0$. So $H_d$ has generators
\begin{equation}\label{hg0}
\{x_1,\dots,x_d\}\cup\{s_1,\dots, s_{d-1}\}
\end{equation}
subject to the following relations for all admissible indices:
\begin{align}
\label{EPoly}
x_rx_s&=x_sx_r;
\\\label{EDAHA}
s_r x_{r+1} &= x_r s_r + 1,\hspace{10.5mm} s_r x_s = x_s s_r \hspace{1.5mm} 
\text{ if $s \neq r,r+1$};\\
\label{ECoxeter1}
s_r^2&=1;\\\label{ECoxeter2}
s_rs_{r+1}s_r&=s_{r+1}s_rs_{r+1},
\qquad
s_rs_t=s_ts_r\hspace{3mm} \text{if $|r-t|>1$}.
\end{align}
Given $\La\in P_+$ of level $\ell$,
let $H_d^\La$ be the cyclotomic quotient from (\ref{ECHA}).
The elements
$x_1^{m_1} \cdots x_d^{m_d} w$
for $0 \leq m_1,\dots,m_d < \ell$ and $w \in S_d$
give a basis for $H_d^\La$; 
see e.g. \cite[Theorem 7.5.6]{Kbook}. 
Hence:
\begin{equation}\label{degdim}
\dim H_d^\La = \ell^d d!.
\end{equation} 
For example, if $\ell=1$  then $H_d^\La\cong F S_d$. 

Let $M$ be a finite dimensional $H_d^\La$-module. By \cite[Lemma 7.1.2]{Kbook}, the eigenvalues of each $x_r$ on $M$ belong to 
$I \subseteq F$. So 
$M$ decomposes as the direct sum 
$M = \bigoplus_{\bi \in I^d} M_\bi$
of its {\em weight spaces} 
$$
M_\bi := \{v \in M\:|\:(x_r-i_r)^N v = 0
\text{ for all $r=1,\dots,d$ and $N \gg 0$}\}.
$$
Note also by \cite[Lemma 2.2.1]{Kbook} that
\begin{equation}\label{ESpan}
s_r(M_\bi)\subseteq M_\bi+M_{s_r\cdot\bi}
\end{equation}
for each $1 \leq r < d$.
Considering the weight space 
decomposition of the regular module, we deduce that 
there is a system 
$\{e(\bi)\:|\:\bi \in I^d\}$
of mutually orthogonal idempotents in $H_d^\La$
such that $e(\bi) M = M_\bi$ for each finite dimensional module $M$. 
In fact, each $e(\bi)$
lies in the commutative subalgebra
generated by $x_1,\dots,x_d$. 
All but finitely many of the $e(\bi)$'s are zero,
and their sum is the identity element in $H_d^\La$. 

By \cite[Theorem 1]{cyclo}, the center $Z(H_d^\La)$ 
consists of all symmetric polynomials in $x_1,\dots,x_d$. So,
given also $\alpha \in Q_+$ of height $d$, the
idempotent $e_{\alpha}$ from (\ref{bal})
is either zero or it is a primitive central idempotent
in $H_d^\La$.
This means that the algebra $H_\alpha^\La := e_{\alpha} H_d^\La$ 
from (\ref{fe})
is either zero or it is a block of  $H_d^\La$.
The subalgebra of $H^\La_\al$ generated by (the images of) $x_1,\dots,x_d$ will be denoted 
$\Pol_\al^\La$. 
Note $\Pol_\al^\La e(\bi)$ is an 
algebra with identity element
$e(\bi)$. If $x\in \Pol^\La_\al$ is such that $xe(\bi)$ is a unit in 
$\Pol_\al^\La e(\bi)$, we write $x^{-1}e(\bi)$ for its inverse in 
$\Pol_\al^\La e(\bi)$ (interpreted as $0$ if $e(\bi) = 0$).
For example, set 
\begin{equation}
x_{r,s}:=x_r-x_s,
\end{equation}
and let $\bi\in I^\al$ be 
such that $i_r\neq i_s$. 
Then $x_{r,s}^{-1}e(\bi)$ makes sense. 

\subsection{\boldmath Intertwining elements $\phi_r$} 
We now introduce certain remarkable elements of $H_\al^\La$ called {\em intertwining elements}: let
\begin{equation}
\phi_r:=s_r+\sum_{\substack{\bi\in I^\al\\i_r\neq i_{r+1}}}x_{r,r+1}^{-1}e(\bi)
+\sum_{\substack{\bi\in I^\al\\i_r= i_{r+1}}}e(\bi)
\end{equation}
for $1 \leq r < d$.
This is a slightly modified version of the usual intertwining element
as in \cite[\S2]{Ro} or \cite[(3.19)]{Kbook}:
\begin{equation}\label{ERelation}
\theta_r:=s_rx_{r,r+1}+1.
\end{equation}
The elements $\theta_r$ have the following nice properties 
\cite[\S3.8]{Kbook} (cf. \cite[Proposition 5.2]{Lu}):
\begin{align}
\label{EOldI4}
\theta_r^2&=
1-x_{r,r+1}^{2};
\\
\label{EOldI1}
\theta_rx_{r+1}&=
x_r\theta_r,\ \
x_{r+1}\theta_r=\theta_rx_r,\ \
\theta_rx_s=x_s\theta_r\  \text{if $s\neq r,r+1$};\\
\label{EOldI6}
\theta_r\theta_{r+1}\theta_r&=\theta_{r+1}\theta_r\theta_{r+1},
\hspace{20mm}\theta_r\theta_s=\theta_s\theta_r\:\, \text{if $|r-s|>1$}.
\end{align}
The elements $\phi_r$ inherit similar properties:

\begin{Lemma}\label{int}
The intertwining elements satisfy the following relations
for all $\bi \in I^\alpha$ and admissible $r, s$:
\begin{align}
\label{EI0}
\phi_re(\bi)&=e(s_r{\cdot}\bi)\phi_r;\\
\label{EI1}
\phi_rx_s&=x_s\phi_r\hspace{49.8mm}\text{if $s\neq r,r+1$};\\
\label{EI5}
\phi_r\phi_s&=\phi_s\phi_r\hspace{49.8mm}\text{if $|r-s|>1$};\\
\label{EI2}
\phi_rx_{r+1}e(\bi)&=\begin{cases}
x_r\phi_re(\bi) & \text{if $i_r\neq i_{r+1}$},\\
(x_r\phi_r+1-x_{r,r+1})e(\bi) \hspace{12.6mm}&\text{if $i_r= i_{r+1}$};
\end{cases}\\
\label{EI3}
\hspace{5.5mm}x_{r+1}\phi_re(\bi)&=\begin{cases}
\phi_rx_re(\bi) & \text{if $i_r\neq i_{r+1}$},\\
(\phi_rx_r+1-x_{r,r+1})e(\bi)\hspace{12.6mm} &\text{if $i_r= i_{r+1}$};
\end{cases}\\
\label{EI4}
\phi_r^2e(\bi)&=\begin{cases}
(1-x_{r,r+1}^{-2})e(\bi) & \text{if $i_r\neq i_{r+1}$},\\
2\phi_re(\bi) \hspace{38.5mm}&\text{if $i_r= i_{r+1}$};
\end{cases}\\
\label{EI6}
\phi_{r}\phi_{r+1}\phi_{r}e(\bi)
&=\begin{cases}
(\phi_{r+1}\phi_{r}\phi_{r+1}
+\phi_{r}-\phi_{r+1})e(\bi) & \text{if $i_r=i_{r+2}=i_{r+1}$},\\
(\phi_{r+1}\phi_{r}\phi_{r+1}+z_r)
e(\bi)
& \text{if $i_r=i_{r+2}\neq i_{r+1}$},\\
\phi_{r+1}\phi_{r}\phi_{r+1}e(\bi) & \text{otherwise},
\end{cases}
\end{align}
where
$z_r$ denotes
$(x_{r,r+1}^{-1} - x_{r+1,r+2}^{-1})
(x_{r,r+1}^{-1} x_{r+1,r+2}^{-1} - x_{r,r+1}^{-1} - x_{r+1,r+2}^{-1})$.
\end{Lemma}

\begin{proof}
To see (\ref{EI0}), it suffices to prove that its left hand side 
and right hand side
act in the same way on $M_\bj$ for
any finite dimensional $H_\al^\La$-module $M$
and $\bj \in I^\alpha$.
If $j_r = j_{r+1}$ then $\phi_r M_\bj \subseteq M_{s_r{\cdot}\bj}$
by (\ref{ESpan}).
If $j_r \neq j_{r+1}$ then $\phi_r e(\bj) = \theta_r x_{r,r+1}^{-1} e(\bj)$,
hence $\phi_r M_\bj = \theta_r M_\bj$
which is contained in $M_{s_r{\cdot}\bj}$ by (\ref{EOldI1}).
Hence in any case $e(s_r {\cdot} \bi) \phi_r$ maps
$M_\bj$ to zero unless $\bj = \bi$, and it maps
$v \in M_\bi$ to $\phi_r v$.
This is the same as the action of $\phi_r e(\bi)$, hence (\ref{EI0}) is checked.
The properties (\ref{EI1}) and (\ref{EI5}) are clear, and
the properties (\ref{EI2}) and (\ref{EI3}) come easily from (\ref{EOldI1}). 
For (\ref{EI4}), if $i_r=i_{r+1}$ then we have that
$\phi_r^2e(\bi)=(s_r+1)^2e(\bi)=(2+2s_r)e(\bi)=2\phi_re(\bi)$.
Now suppose that $i_r\neq i_{r+1}$. Then, using (\ref{EOldI1}) and 
(\ref{EOldI4}), we have 
$$\phi_r^2e(\bi)
=\theta_rx_{r,r+1}^{-1}\theta_rx_{r,r+1}^{-1}e(\bi)
=
-\theta_r^2x_{r,r+1}^{-2}e(\bi)
=(1-x_{r,r+1}^{-2})e(\bi).
$$
It remains to check (\ref{EI6}). For this, 
let us stop writing $e(\bi)$ on the right of 
all expressions (but of course remember it is there). 
Assume without loss of generality that $r=1$ and denote $i:=i_1, j:=i_2, k:=i_3$.
We consider five cases:

\subsubsection*{Case 1: $i,j,k$ all distinct} Using (\ref{EOldI1}), 
(\ref{EOldI6}) and (\ref{EI0}), we have:
\begin{align*}
\phi_2\phi_1\phi_2&=\theta_2x_{2,3}^{-1}\theta_1x_{1,2}^{-1}\theta_2x_{2,3}^{-1}
=\theta_2\theta_1\theta_2x_{1,2}^{-1}x_{1,3}^{-1}x_{2,3}^{-1}
\\&=\theta_1\theta_2\theta_1x_{2,3}^{-1}x_{1,3}^{-1}x_{1,2}^{-1}
=\theta_1x_{1,2}^{-1}\theta_2x_{2,3}^{-1}\theta_1x_{1,2}^{-1}=\phi_1\phi_2\phi_1.
\end{align*}
 
\subsubsection*{Case 2: $i=j\neq k$} Using (\ref{EI0}), we see that $\phi_2\phi_1\phi_2=\phi_1\phi_2\phi_1$ is equivalent to $(s_2+1)\phi_1\phi_2=\phi_1\phi_2(s_1+1)$, or $s_2\phi_1\phi_2=\phi_1\phi_2s_1$. Also the
relations in $H^\La_\al$ give that
$x_{1,3}^{-1} s_1 = s_1 x_{2,3}^{-1} + x_{2,3}^{-1} x_{1,3}^{-1}$
and $x_{2,3}^{-1} s_1 = s_1 x_{1,3}^{-1} - x_{2,3}^{-1} x_{1,3}^{-1}$
(remembering the idempotent $e(\bi)$ implicitly appears 
on the right so all the inverses here make sense).
Now apply (\ref{EOldI1}), (\ref{EOldI6}) and these two
relations to commute all $x$'s to the right and show that both sides are equal to $s_2s_1s_2+s_2s_1x_{2,3}^{-1}+s_2x_{2,3}^{-1}x_{1,3}^{-1}+x_{1,3}^{-1}$.
 
\subsubsection*{Case 3: $i\neq j= k$} This case is similar to Case 2.
 
\subsubsection*{Case 4: $i=k\neq j$} Using (\ref{EOldI1}), (\ref{EOldI6}) 
and (\ref{EI0}), we get
  \begin{align*}
\phi_2\phi_1\phi_2&=\theta_2x_{2,3}^{-1}(s_1+1)\theta_2x_{2,3}^{-1}
=\theta_2x_{2,3}^{-1}s_1\theta_2x_{2,3}^{-1}-\theta_2^2x_{2,3}^{-2}
\\&=\theta_2(s_1x_{1,3}^{-1}-x_{2,3}^{-1}x_{1,3}^{-1})\theta_2x_{2,3}^{-1}-(1-x_{2,3}^2)x_{2,3}^{-2}
\\&=\theta_2s_1\theta_2x_{1,2}^{-1}x_{2,3}^{-1}+\theta_2^2x_{2,3}^{-2}x_{1,2}^{-1}
-x_{2,3}^{-2}+1
\\&=(s_2x_{2,3}+1)s_1\theta_2x_{1,2}^{-1}x_{2,3}^{-1}+(1-x_{2,3}^2)x_{2,3}^{-2}x_{1,2}^{-1}
-x_{2,3}^{-2}+1\hspace{9.6mm}\\
&=s_2(s_1x_{1,3}+1)\theta_2x_{1,2}^{-1}x_{2,3}^{-1}+s_1\theta_2x_{1,2}^{-1}x_{2,3}^{-1}
+(x_{2,3}^{-2}-1)(x_{1,2}^{-1}-1)\\
\phantom{\phi_2\phi_1\phi_2}&=s_2s_1\theta_2x_{2,3}^{-1}+s_2\theta_2x_{1,2}^{-1}x_{2,3}^{-1}
+s_1\theta_2x_{1,2}^{-1}x_{2,3}^{-1}
+(x_{2,3}^{-2}-1)(x_{1,2}^{-1}-1)\!\!\!
\\&=s_2s_1(s_2x_{2,3}+1)x_{2,3}^{-1}+s_2(s_2x_{2,3}+1)x_{1,2}^{-1}x_{2,3}^{-1}
\\&\qquad+s_1(s_2x_{2,3}+1)x_{1,2}^{-1}x_{2,3}^{-1}
+(x_{2,3}^{-2}-1)(x_{1,2}^{-1}-1)
\\&=s_2s_1s_2+s_2s_1x_{2,3}^{-1}
+s_1s_2x_{1,2}^{-1}
+s_2x_{1,2}^{-1}x_{2,3}^{-1}
+s_1x_{1,2}^{-1}x_{2,3}^{-1}
\\&\qquad+x_{1,2}^{-1}+(x_{2,3}^{-2}-1)(x_{1,2}^{-1}-1).\\\intertext{Similarly, we have}
\phi_1\phi_2\phi_1&=
s_1s_2s_1
+s_1s_2x_{1,2}^{-1}
+s_2s_1x_{2,3}^{-1}
+s_1x_{2,3}^{-1}x_{1,2}^{-1}
+s_2x_{2,3}^{-1}x_{1,2}^{-1}
\\&\qquad+x_{2,3}^{-1}+(x_{1,2}^{-2}-1)(x_{2,3}^{-1}-1),
\end{align*}
and (\ref{EI6}) now follows. 

\subsubsection*{Case 5: $i= j= k$} This case follows since
$(s_1+1)(s_2+1)(s_1+1)+s_2=(s_2+1)(s_1+1)(s_2+1)+s_1$.  
\end{proof}

\subsection{\boldmath Khovanov-Lauda generators of $H^\La_\al$ in the degenerate case}\label{SKLGens}
For each $r=1,\dots,d$,
the elements
\begin{equation}\label{jon}
y_r :=
\sum_{\bi \in I^\al} (x_r-i_r) e(\bi)
\end{equation}
are nilpotent elements of the commutative algebra
$\Pol_\al^\La$.
So we are in the situation of (\ref{conv})
and get a homomorphism $F[[y_1,\dots,y_d]] \rightarrow \Pol_\al^\La$
mapping each $y_r \in F[[y_1,\dots,y_d]]$ 
to the element (\ref{jon}).
We are often going to abuse notation by using the same symbol 
for a power series $f \in F[[y_1,\dots,y_d]]$ and 
for its image in $\Pol^\La_\al$ under this homomorphism.

For $1\leq r<d$ and $\bi\in I^\al$, we define power series
$p_r(\bi) \in F[[y_r,y_{r+1}]]$ by setting
\begin{equation}
\label{EP}
p_r(\bi):=\begin{cases}
1 & \text{if $i_r=i_{r+1}$},\\
(i_{r}-i_{r+1}+y_{r}-y_{r+1})^{-1} & \text{if $i_r\neq i_{r+1}$}.
\end{cases}
\end{equation}
The following facts are easy to check:
\begin{align}
{^{s_r}}p_r(s_r{\cdot}\bi)&=-p_r(\bi)\hspace{19.6mm}\text{if}\ i_r\neq i_{r+1};\label{EP1}
\\
{^{s_r}}p_{r+1}(s_r{\cdot}\bi)&={^{s_{r+1}}}p_r(s_{r+1}{\cdot}\bi)\qquad\text{for any}\ \bi.\label{EP3}
\end{align}
Note also 
for all $1\leq r<d$ and $\bi\in I^\al$ that
\begin{align}\label{EPE}
p_r(\bi)e(\bi)&=\begin{cases}
e(\bi) & \text{if $i_r=i_{r+1}$},\\
x_{r,r+1}^{-1}e(\bi) & \text{if $i_r\neq i_{r+1}$};
\end{cases}
\\
\label{EPPhi}
\phi_r&=\sum_{\bi\in I^\al}(s_r+p_r(\bi))e(\bi).  
\end{align}

Also make an arbitrary but henceforth fixed
choice of (invertible) elements
$q_r(\bi) \in F[[y_r,y_{r+1}]]$ 
with the following properties:
\begin{align}\label{EQProp1}
q_r(\bi) &=
1+y_{r+1}-y_r\hspace{26.5mm}\text{if $i_r = i_{r+1}$};\\
q_r(\bi){^{s_r\!}}q_r(s_r{\cdot}\bi) &=
\left\{
\begin{array}{ll}
1-p_r(\bi)^2&\text{if $i_r \nslash i_{r+1}$},\\
(1-p_r(\bi)^2)/(y_{r+1}-y_r)&\text{if $i_r \rightarrow i_{r+1}$},\\
(1-p_r(\bi)^2)/(y_{r}-y_{r+1})&\text{if $i_r \leftarrow i_{r+1}$},\\
\frac{1-p_r(\bi)^2}{(y_{r+1}-y_r)(y_{r}-y_{r+1})}\hspace{5.5mm}&\text{if $i_r \rightleftarrows 
i_{r+1}$};
\end{array}\right.
\label{EQProp4}\\
{^{s_r\!}}q_{r+1}(s_{r+1}s_r{\cdot}\bi)&={^{s_{r+1}\!}}q_r(s_r s_{r+1}{\cdot}\bi) \hspace{23.2mm}\text{for any $\bi$}.\label{EQProp2}
\end{align}
Note in the fractions on the right hand side of 
(\ref{EQProp4}) that the numerator is divisible by the denominator
in $F[[y_r,y_{r+1}]]$,
so this makes sense.
Moreover, it is always possible to choose such power series $q_r(\bi)$. 
For instance, one could take
\begin{align}
q_r(\bi) &:= 
\left\{
\begin{array}{ll}
1+y_{r+1}-y_r&\text{if $i_r = i_{r+1}$},\\
1 - p_r(\bi)&\text{if $i_r \nslash i_{r+1}$},\\
(1 - p_r(\bi)^2) / (y_{r+1}-y_r)&\text{if $i_r \rightarrow i_{r+1}$},\\
1&\text{if $i_r \leftarrow i_{r+1}$},\\
(1-p_r(\bi)) / (y_{r+1}-y_r)&\text{if $i_r \rightleftarrows i_{r+1}$},
\end{array}
\right.
\end{align}
although we do not want to restrict ourselves to this particular choice.

Now, the {\em Khovanov-Lauda generators} of $H^\La_\alpha$ are the elements
\begin{equation}\label{klg}
\{e(\bi)\:|\:\bi\in I^\al\}\cup\{y_1,\dots,y_d\}\cup\{\psi_1,\dots,\psi_{d-1}\},
\end{equation}
where 
$y_r$ is the element defined by (\ref{jon}) and
\begin{align}
\psi_r&:=\sum_{\bi\in I^\al}\phi_rq_r(\bi)^{-1}e(\bi)=
\sum_{\bi\in I^\al}(s_r+p_r(\bi))q_r(\bi)^{-1}e(\bi).
\label{ECoxKL}
\end{align}

\begin{Theorem}\label{dogen}
The elements (\ref{klg}) of $H^\La_\alpha$ 
 satisfy the defining relations (\ref{ERCyc})--(\ref{R7})
of the cyclotomic Khovanov-Lauda algebra.
\end{Theorem}

\begin{proof}
We have by (\ref{ECHA}) that
$\prod_{i\in I}(x_1-i)^{(\La,\al_{i})}=0$. Moreover, if $i\neq i_1$ then $(x_1-i)e(\bi)$ is invertible in $\Pol_\al^\La e(\bi)$. So 
$(x_1-i_{1})^{(\La,\al_{i_1})}e(\bi)=0$, which immediately implies (\ref{ERCyc}). 
We already know that (\ref{R1}) holds.
The relations (\ref{R2PsiE}), (\ref{R3Y}), 
(\ref{R3YPsi}) and (\ref{R3Psi}) follow using 
the fact that $\Pol_\al^\La$ is commutative and 
the properties (\ref{EI0}), (\ref{EI1}) and (\ref{EI5}) of the intertwining elements. 

For (\ref{R5}), in view of (\ref{R2PsiE}), we have
\begin{equation}\label{ESunday}
y_{r+1} \psi_r e(\bi)=y_{r+1} e(s_r{\cdot}\bi)\psi_r e(\bi)=(x_{r+1}-i_{r})\phi_rq_r(\bi)^{-1}e(\bi).
\end{equation}
If $i_r\neq i_{r+1}$, 
this equals  
$\phi_rq_r(\bi)^{-1}(x_{r}-i_{r})e(\bi)=\psi_ry_re(\bi)$ by (\ref{EI3}).
If $i_r=i_{r+1}$, then (\ref{ESunday}) gives 
\begin{align*}
y_{r+1} \psi_r e(\bi) &=(x_{r+1}-i_{r})(s_r+1)(1-x_{r,r+1})^{-1}e(\bi)\\
&=\big((s_r+1)(x_{r}-i_{r})+1-x_{r,r+1}\big)(1-x_{r,r+1})^{-1}e(\bi)\\&=(\psi_ry_r+1)e(\bi).
\end{align*}
The proof of (\ref{R6}) is similar. 

For (\ref{R4}), we have that
\begin{equation}\label{ESat}
\psi_r^2e(\bi) =\phi_rq_r(s_r{\cdot}\bi)^{-1}\psi_re(\bi). 
\end{equation}
If $i_r = i_{r+1}$, the relations in $H^\La_\alpha$
give easily that
$(s_r+1)(1+x_{r,r+1}) = (1-x_{r,r+1})(s_r-1)$, hence we get from (\ref{ESat})
that
\begin{align*}
\psi_r^2 e(\bi) &=
(s_r+1)(1-x_{r,r+1})^{-1}(s_r+1)(1-x_{r,r+1})^{-1}e(\bi)
\\&=(s_r+1)(s_r-1)(1+x_{r,r+1})^{-1}(1-x_{r,r+1})^{-1}e(\bi)=0.
\end{align*}
Now suppose that $i_r \neq i_{r+1}$. Note as we have now checked the
relations (\ref{R3YPsi}), (\ref{R6}) and (\ref{R5}),
the identity (\ref{EDDR}) holds in the present situation.
Using (\ref{EDDR}), (\ref{EI4}) and (\ref{EPE}), the equation
(\ref{ESat}) becomes
\begin{align*}
\psi_r^2 e(\bi) &=
\phi_r\psi_r({^{s_r\!}}q_r(s_r{\cdot}\bi))^{-1}e(\bi)=\phi_r^2q_r(\bi)^{-1}\left({^{s_r\!}}q_r(s_r{\cdot}\bi)\right)^{-1}e(\bi)\\
&=
(1 - p_r(\bi)^2)  \left(q_r(\bi) {^{s_r\!}} q_r(s_r{\cdot}\bi)\right)^{-1}
e(\bi).
\end{align*}
Using (\ref{EQProp4}), this simplifies to give
the right hand side of (\ref{R4}).

Finally we prove (\ref{R7}). Let us stop writing $e(\bi)$ at the right of all expressions.  
Assume without loss of generality that $r=1$, $d=3$,  and denote $i := i_1, j := i_2$, $k := i_3$.

\subsubsection*{Case 1: $i,j,k$ all distinct}
Using (\ref{EDDR}) and (\ref{R2PsiE}), we get
\begin{align*}
\psi_1\psi_2\psi_1&=\phi_1q_1(jki)^{-1}\psi_2\psi_1=\phi_1\psi_2\psi_1
\left({^{s_1s_2}}q_1(jki)\right)^{-1}
\\&
=\phi_1\phi_2q_2(jik)^{-1}\psi_1\left({^{s_1s_2}}q_1(jki)\right)^{-1}\\
&=\phi_1\phi_2\psi_1\left({^{s_1\!}}q_2(jik){^{s_1s_2}}q_1(jki)\right)^{-1}
\\
&=\phi_1\phi_2\phi_1\left(q_1(ijk){^{s_1\!}}q_2(jik){^{s_1s_2}}q_1(jki)\right)^{-1}.
\end{align*}
Similarly,
$$
\psi_2\psi_1\psi_2=
\phi_2\phi_1\phi_2\left(q_2(ijk){^{s_2}}q_1(ikj){^{s_2s_1\!}}q_2(kij)\right)^{-1}.
$$
Now, $\phi_1\phi_2\phi_1=\phi_2\phi_1\phi_2$ by (\ref{EI6}), and also $q_2(ijk)={^{s_1s_2}}q_1(jki)$, ${^{s_2}}q_1(ikj)={^{s_1\!}}q_2(jik)$ and 
${^{s_2s_1\!}}q_2(kij)=q_1(ijk)$ by (\ref{EQProp2}). Thus  (\ref{R7}) holds.

\subsubsection*{Case 2: $i=j\neq k$}
As in the previous case, we get that
\begin{equation}\label{EPicknic}
\psi_2\psi_1\psi_2=\phi_2\phi_1\phi_2\left(q_2(iik){^{s_2}}q_1(iki){^{s_2s_1\!}}q_2(kii)\right)^{-1}.
\end{equation}
On the other hand, by (\ref{EDDR}) and (\ref{EQProp2}),
\begin{align*}
\psi_1\psi_2\psi_1&=\phi_1q_1(iki)^{-1}\psi_2\psi_1\\
&=\phi_1\psi_2\psi_1({^{s_1s_2}}q_1(iki))^{-1}+
\phi_1\psi_2\partial_1\!\left(({^{s_2}}q_1(iki))^{-1}\right)
\\
&=\phi_1\phi_2q_2(iik)^{-1}\psi_1({^{s_1s_2}}q_1(iki))^{-1}\!+\phi_1\phi_2q_2(iik)^{-1}\partial_1\!\left(({^{s_2}}q_1(iki))^{-1}\right)
\end{align*}\begin{align*}
\phantom{\psi_1\psi_2\psi_1}
&=\phi_1\phi_2\psi_1\left({^{s_1\!}}q_2(iik){^{s_1s_2}}q_1(iki)\right)^{-1}
\!+\phi_1\phi_2\partial_1\!\left(q_2(iik)^{-1}\right)\!({^{s_1s_2}}q_1(iki))^{-1}\\
&\qquad+\phi_1\phi_2q_2(iik)^{-1}\partial_1\!\left(({^{s_2}}q_1(iki))^{-1}\right).
\\
&=\phi_1\phi_2\phi_1\left(q_1(iik){^{s_1\!}}q_2(iik){^{s_1s_2}}q_1(iki)\right)^{-1}
\\&\qquad+\phi_1\phi_2
q_2(iik)^{-1}
\partial_1\!\left(q_2(iik)^{-1}
+({^{s_1\!}}q_2(iik))^{-1}\right).
\end{align*}
The first term of the last expression equals the right hand side of (\ref{EPicknic}) by (\ref{EI6}) and (\ref{EQProp2}). It remains to observe
that the second term of the last expression is zero,
as $\partial_1 f = 0$ for any $f$ with
${^{s_1\!}f}=f$.

\subsubsection*{Case 3: $i\neq j=k$} This case is similar to Case 2. 

\subsubsection*{Case 4: $i=k\neq j$} As in the previous cases we compute:
\begin{align*}
\psi_1\psi_2\psi_1&=\phi_1\phi_2\phi_1
\left(q_1(iji){^{s_1\!}}q_2(jii){^{s_1s_2}}q_1(jii)\right)^{-1}
\\&\qquad+(1-x_{1,2}^{-2})q_1(iji)^{-1}\left({^{s_1}}\partial_2(q_1(jii)^{-1})\right),\\
\psi_2\psi_1\psi_2&=
\phi_2\phi_1\phi_2\left(q_2(iji){^{s_2}}q_1(iij){^{s_2s_1\!}}q_2(iij)\right)^{-1}
\\&\qquad+(1-x_{2,3}^{-2})q_2(iji)^{-1}\left({^{s_2}}\partial_1(q_2(iij)^{-1})\right).
\end{align*}
So by (\ref{EI6}), (\ref{EQProp2}), and (\ref{EPE}), we get that
$\psi_1\psi_2\psi_1-\psi_2\psi_1\psi_2=A+B-C$, where 
\begin{align*}
A&=\big(p_{1}(iji)-p_{2}(iji)\big)\big(p_{1}(iji)p_{2}(iji)-p_{1}(iji)
-p_{2}(iji)\big)\\
&\qquad\qquad\qquad\times \left(q_1(iji){^{s_1\!}}q_2(jii)q_2(iji)\right)^{-1},
\\B&=(1-p_{1}(iji)^2)q_1(iji)^{-1}\left({^{s_1}}\partial_2(q_1(jii)^{-1})\right),
\\C&=(1-p_{2}(iji)^{2})q_2(iji)^{-1}\left({^{s_2}}\partial_1(q_2(iij)^{-1})\right).
\end{align*}
By substituting 
$p_1(iji) = (i-j+y_1-y_2)^{-1}$, $p_2(iji)=(j-i+y_2-y_3)^{-1}$
and putting over a common denominator, it is straightforward to check
the following power series identity:
\begin{equation}\label{PID1}
p_1(iji)+p_2(iji) = (y_1-y_3) p_1(iji) p_2(iji).
\end{equation}
Hence:
$$
\frac{p_1(iji)p_2(iji)-p_1(iji)-p_2(iji)}{1+y_3-y_1}
=
\frac{p_1(iji)+p_2(iji)}{y_1-y_3}.
$$
Using this and noting ${^{s_1\!}}q_2(jii) = 1 +y_3-y_1$,
we deduce that
\begin{equation}\label{A}
A = \frac{(p_1(iji)^2-p_2(iji)^2) q_1(iji)^{-1} q_2(iji)^{-1}}{y_1-y_3}.
\end{equation}
Using the definition (\ref{EPartial}) and the property (\ref{EQProp4}), 
we have that
$$
{^{s_1}}\partial_2(q_1(jii)^{-1})=\frac{q_2(iji)^{-1}-({^{s_1\!}}q_1(jii))^{-1}}{y_1-y_3}.
$$
Substituting this into $B$, we get that
\begin{equation}\label{B}
B=
\frac{(1-p_1(iji)^2) \left((q_1(iji)^{-1} q_2(iji)^{-1} - (q_1(iji) {^{s_1\!}}q_1(jii))^{-1}\right)}{y_1-y_3}.
\end{equation}
Similarly,
\begin{equation}\label{C}
C
=\frac{(1-p_{2}(iji)^2)\left(q_1(iji)^{-1}q_2(iji)^{-1}-(q_2(iji) {^{s_2}}q_2(iij))^{-1}\right)}{y_1-y_3}.
\end{equation}
The equations (\ref{A}), (\ref{B}) and (\ref{C}) easily give that
$A+B-C$ equals
$$
\frac{(1-p_2(iji)^2)(q_2(iji){^{s_2}}q_2(iij))^{-1}-(1-p_1(iji)^2)(q_1(iji){^{s_1\!}}q_1(jii))^{-1}}{y_1-y_3}.
$$
Finally by (\ref{EQProp4}) this is $0$
if $i \nslash j$, 
$1$ if $i \rightarrow j$, $-1$ if $i \leftarrow j$
or $-2y_2+y_1+y_3$ if $i \rightleftarrows j$.
This imples (\ref{R7}).

\subsubsection*{Case 5: $i=j=k$} We leave this case as an exercise to the 
reader.
\end{proof} 
  
\subsection{\boldmath Degenerate Hecke generators of $R^\La_\al$}
Let $R^\La_\alpha$ be the cyclotomic Khovanov-Lauda algebra
from $\S$\ref{sckl}.
Using the homomorphism (\ref{Rhom}),
we can regard the power series
$p_r(\bi)$ from (\ref{EP}) 
as elements of $R_\al^\La$.
Similarly, the power series
$q_r(\bi)$ satisfying (\ref{EQProp1})--(\ref{EQProp2})
that were chosen in $\S$\ref{SKLGens}
give rise to elements $q_r(\bi) \in R_\al^\La$.
The {\em degenerate Hecke generators} of $R_\al^\La$ are the elements
\begin{equation}\label{hg}
\{x_1,\dots,x_d\}\cup\{s_1,\dots,s_{d-1}\}
\end{equation}
where
\begin{align}
 x_r&:=\sum_{\bi\in I^\al}
 (y_r+i_r)e(\bi),\label{ET1}\\
 s_r&:=\sum_{\bi\in I^\al}
 (\psi_rq_r(\bi)-p_r(\bi))e(\bi).
 \label{ET2}
\end{align}
We note by (\ref{R4})
and (\ref{EQProp4}) that
\begin{equation}\label{EUs}
\psi_r^2q_r(\bi){^{s_r\!}}q_r(s_r{\cdot}\bi)e(\bi)=(1-p_r(\bi)^2)e(\bi)
\end{equation}
for all $\bi \in I^\alpha$ and $1 \leq r < d$.
The following result is the key technical result needed to complete the
proof of our Main Theorem in the degenerate case.

\begin{Theorem}\label{thm2}
The elements (\ref{hg})
of
$R^\La_\alpha$ 
satisfy the defining relations (\ref{EPoly})--(\ref{ECoxeter2})
of the degenerate affine Hecke algebra $H_d$.
\end{Theorem}

\begin{proof}
The polynomial relation (\ref{EPoly}) is obvious: the 
$x_r$'s commute because the $y_r$'s and $e(\bi)$'s do. 
The mixed relation (\ref{EDAHA}) is clear for 
$s\neq r,r+1$.
For the remaining mixed relation, it suffices to show that
$$
(s_r x_{r+1} - x_r s_r) e(\bi) = e(\bi)
$$ 
for 
every $\bi \in I^\alpha$.
Expand the definitions
(\ref{ET1})--(\ref{ET2}) using (\ref{R2PsiE}) gives:
\begin{align*}
s_r x_{r+1} e(\bi)
&=
(\psi_r q_r(\bi) - p_r(\bi)) (y_{r+1} + i_{r+1}) e(\bi),\\
x_r s_r e(\bi) &=
x_r (\psi_r q_r(\bi) - p_r(\bi)) e(\bi)\\&=
x_r e(s_r{\cdot}\bi) \psi_r q_r(\bi) e(\bi) - x_r e(\bi) p_r(\bi) e(\bi)\\
&=
(y_r + i_{r+1}) \psi_r q_r(\bi) e(\bi) - (y_r+i_r) p_r(\bi) e(\bi).
\end{align*}
Hence
$$
(s_r x_{r+1} - x_r s_r)e(\bi)
=
(i_r - i_{r+1} + y_r - y_{r+1}) p_r(\bi) e(\bi)
+ (\psi_r y_{r+1}-y_r \psi_r) q_r(\bi) e(\bi).
$$
Applying (\ref{R6}), (\ref{EP}) and (\ref{EQProp1}),
we have that
$(\psi_r y_{r+1}-y_r \psi_r)e(\bi) = e(\bi)$,
$p_r(\bi) = 1$ and $q_r(\bi) = 1 + y_{r+1}-y_r$
if $i_r = i_{r+1}$, or
$(\psi_r y_{r+1} - y_r \psi_r) e(\bi) = 0$
and $p_r(\bi) = (i_r - i_{r+1} + y_r - y_{r+1})^{-1}$
if $i_r \neq i_{r+1}$.
Making these substitutions 
gives easily that $(s_r x_{r+1}-x_r s_r) e(\bi) = e(\bi)$
as required.

Next we check the quadratic relation (\ref{ECoxeter1}).
For this we need to show that $$
s_r^2 e(\bi) = e(\bi)
$$ for each $\bi \in 
I^\alpha$.
Expanding the definition (\ref{ET2}),
we get that
\begin{align*}
s_r^2 e(\bi) &= 
s_r (\psi_r q_r(\bi) - p_r(\bi)) e(\bi)
=s_r e(s_r{\cdot}\bi)\psi_r q_r(\bi) e(\bi) - s_r e(\bi)p_r(\bi) e(\bi)\\
&= 
(\psi_r q_r(s_r{\cdot}\bi) - p_r(s_r{\cdot}\bi)) \psi_r q_r(\bi) e(\bi)
- (\psi_rq_r(\bi) - p_r(\bi))p_r(\bi) e(\bi)\\
&= 
\left(\psi_r q_r(s_r{\cdot}\bi)\psi_r q_r(\bi)  - p_r(s_r{\cdot}\bi)\psi_r q_r(\bi)  
- \psi_rq_r(\bi)p_r(\bi) + p_r(\bi)^2 \right)e(\bi).
\end{align*}
If $i_r = i_{r+1}$, we use
(\ref{EP}) and (\ref{EQProp1}) to get from this that
$$
s_r^2 e(\bi) = 
\left(\psi_r (1+y_{r+1}-y_r) \psi_r q_r(\bi)
-2 q_r(\bi) + 1\right) e(\bi).
$$
Using (\ref{R6})--(\ref{R5}) to commute $y$'s to the right
and noting that $\psi_r^2 e(\bi) = 0$ by (\ref{R4}),
this easily simplifies to give the desired equation
$s_r^2 e(\bi) = e(\bi)$.
Instead, if $i_r \neq i_{r+1}$, then we again
commute $y$'s to the right
and use
(\ref{EUs}) and (\ref{EP1}) to get that
\begin{align*}
s_r^2 e(\bi) &=
\left((1 - p_r(\bi)^2)
+ \psi_r p_r(\bi) q_r(\bi)
- \psi_r q_r(\bi) p_r(\bi) + p_r(\bi)^2\right) e(\bi)= e(\bi).
\end{align*}
This completes the proof of (\ref{ECoxeter1}).

Finally we need to check the braid relations (\ref{ECoxeter2}).
The commuting braid relation is obvious.
For the length three braid relation,
we assume without loss of generality that $r=1$ and $d=3$, and 
need to show that $$
s_2s_1s_2 e(ijk) = s_1s_2s_1 e(ijk)
$$
for all $i,j,k$.
To simplify notation for the remainder of the proof, we 
stop writing $e(ijk)$ on the right hand side of all expressions, but
remember it is always there.
Expanding the definition (\ref{ET2}) like in the previous paragraph,
$s_2s_1s_2$ and $s_1s_2s_1$
equal
\begin{equation}\label{ELHS}
\begin{split}
&-p_2(ijk)p_1(ijk)p_2(ijk)+ \psi_2q_2(ijk)p_1(ijk)p_2(ijk)
\\
&+ p_2(jik)\psi_1q_1(ijk)p_2(ijk)
-\psi_2q_2(jik)\psi_1q_1(ijk)p_2(ijk)
\\
&+ p_2(ikj)p_1(ikj)\psi_2q_2(ijk)
- \psi_2q_2(ikj)p_1(ikj)\psi_2q_2(ijk)
\\
&- p_2(kij)\psi_1q_1(ikj)\psi_2q_2(ijk)
+ \psi_2q_2(kij)\psi_1q_1(ikj)\psi_2q_2(ijk),
\end{split}
\end{equation}
and 
\begin{equation}\label{ERHS}
\begin{split}
&-p_1(ijk)p_2(ijk)p_1(ijk)
+ \psi_1q_1(ijk)p_2(ijk)p_1(ijk)
\\&+ p_1(ikj)\psi_2q_2(ijk)p_1(ijk)
- \psi_1q_1(ikj)\psi_2q_2(ijk)p_1(ijk)
\\&+ p_1(jik)p_2(jik)\psi_1q_1(ijk)
- \psi_1q_1(jik)p_2(jik)\psi_1q_1(ijk)
\\& - p_1(jki)\psi_2q_2(jik)\psi_1q_1(ijk)
+ \psi_1q_1(jki)\psi_2q_2(jik)\psi_1q_1(ijk),
\end{split}
\end{equation}
respectively.
We have to prove that (\ref{ELHS}) equals (\ref{ERHS}). 
For this we consider five cases. The 
strategy is always to commute all $\psi$'s to the left 
using (\ref{EDDR}) 
then to compare various 
$\psi$-coefficients. 

\subsubsection*{Case 1: $i,j,k$ are all different} By 
(\ref{EDDR}), (\ref{ELHS}) equals 
\begin{align*}
&-p_2(ijk)p_1(ijk)p_2(ijk)+ \psi_2q_2(ijk)p_1(ijk)p_2(ijk)
\\
&+ \psi_1{^{s_1\!}}p_2(jik)q_1(ijk)p_2(ijk)
- \psi_2\psi_1{^{s_1\!}}q_2(jik)q_1(ijk)p_2(ijk)
\\
&+ \psi_2{^{s_2}}p_2(ikj){^{s_2}}p_1(ikj)q_2(ijk)
- (\psi_2^2){^{s_2}}q_2(ikj){^{s_2}}p_1(ikj)q_2(ijk)
\\
&- \psi_1\psi_2{^{s_2s_1\!}}p_2(kij){^{s_2}}q_1(ikj)q_2(ijk)
+ \psi_2\psi_1\psi_2{^{s_2s_1\!}}q_2(kij){^{s_2}}q_1(ikj)q_2(ijk),
\end{align*}
and (\ref{ERHS}) equals
\begin{align*}
&-p_1(ijk)p_2(ijk)p_1(ijk)
+ \psi_1q_1(ijk)p_2(ijk)p_1(ijk)
\\&+ \psi_2{^{s_2}}p_1(ikj)q_2(ijk)p_1(ijk)
- \psi_1\psi_2{^{s_2}}q_1(ikj)q_2(ijk)p_1(ijk)
\\&+ \psi_1{^{s_1\!}}p_1(jik){^{s_1\!}}p_2(jik)q_1(ijk)
- (\psi_1^2){^{s_1\!}}q_1(jik){^{s_1\!}}p_2(jik)q_1(ijk)
\\& - \psi_2\psi_1{^{s_1s_2}}p_1(jki){^{s_1\!}}q_2(jik)q_1(ijk)
+ \psi_1\psi_2\psi_1{^{s_1s_2}}q_1(jki){^{s_1\!}}q_2(jik)q_1(ijk).
\end{align*}
Note that $\psi_2\psi_1\psi_2=\psi_1\psi_2\psi_1$ under our assumptions on $i,j,k$, and the corresponding coefficients are equal to each other 
by (\ref{EQProp2}). 
For the $\psi_1\psi_2$-coefficients, we need to observe by (\ref{EP3}) that
${^{s_2s_1\!}}p_2(kij)=p_1(ijk)$. 
The $\psi_2\psi_1$-coefficients are treated similarly. 
For the $\psi_1$-coefficients, it suffices to prove that   
$$
{^{s_1\!}}p_2(jik)p_2(ijk)
=p_2(ijk)p_1(ijk)+{^{s_1\!}}p_1(jik){^{s_1\!}}p_2(jik), 
$$
which is easily checked by expanding 
the definition (\ref{EP}) and clearing denominators. 
The $\psi_2$-coefficients are handled similarly. 
Finally, the constant term reduces using (\ref{EUs}) 
and the observation that ${^{s_2}}p_1(ikj) = {^{s_1\!}}p_2(jik)$
to checking that
$$
({^{s_2}}p_1(ikj)-p_1(ijk)) p_2(ijk)^2 = 
({^{s_1\!}}p_2(jik)-p_2(ijk))p_1(ijk)^2.
$$
Again this identity follows by an explicit expansion
using (\ref{EP}).

\subsubsection*{Case 2: $i=j\neq k$} 
We have $p_1(iik)=p_2(kii)=1$ by definition.
So, 
using (\ref{EDDR}), the expression (\ref{ELHS}) equals
\begin{equation*}
\begin{split}
&-p_2(iik)^2+ \psi_2q_2(iik)p_2(iik)
+ \psi_1{^{s_1\!}}p_2(iik)q_1(iik)p_2(iik)
\\
&+\partial_1(p_2(iik))q_1(iik)p_2(iik)
- \psi_2\psi_1{^{s_1\!}}q_2(iik)q_1(iik)p_2(iik)
\\
&-\psi_2\partial_1(q_2(iik))q_1(iik)p_2(iik)
+ \psi_2{^{s_2}}p_2(iki){^{s_2}}p_1(iki)q_2(iik)
\\
&- (\psi_2^2){^{s_2}}q_2(iki){^{s_2}}p_1(iki)q_2(iik)
-\psi_1\psi_2{^{s_2}}q_1(iki)q_2(iik)
\\
&+ \psi_2\psi_1\psi_2{^{s_2s_1\!}}q_2(kii){^{s_2}}q_1(iki)q_2(iik).
\end{split}
\end{equation*}
Similarly, using also (\ref{EPR}),
(\ref{ERHS}) becomes
\begin{equation*}
\begin{split}
&-p_2(iik)
+ \psi_1q_1(iik)p_2(iik)
+ \psi_2{^{s_2}}p_1(iki)q_2(iik)
- \psi_1\psi_2{^{s_2}}q_1(iki)q_2(iik)
\\&+\psi_1{^{s_1\!}}p_2(iik)q_1(iik)
+\partial_1(p_2(iik))q_1(iik)
- (\psi_1^2){^{s_1\!}}q_1(iik){^{s_1\!}}p_2(iik)q_1(iik)
\\&- \psi_1\partial_1(q_1(iik)){^{s_1\!}}p_2(iik)q_1(iik)
- \psi_1\partial_1(p_2(iik))q_1(iik)^2
\\& - \psi_2\psi_1{^{s_1s_2}}p_1(iki){^{s_1\!}}q_2(iik)q_1(iik)
-\psi_2\partial_1({^{s_2}}p_1(iki)){^{s_1\!}}q_2(iik)q_1(iik)
\\&- \psi_2{^{s_2}}p_1(iki)\partial_1(q_2(iik))q_1(iik)
+ \psi_1\psi_2\psi_1{^{s_1s_2}}q_1(iki){^{s_1\!}}q_2(iik)q_1(iik)
\\&+ \psi_1\psi_2\partial_1({^{s_2}}q_1(iki)){^{s_1\!}}q_2(iik)q_1(iik)
+ \psi_1\psi_2{^{s_2}}q_1(iki)\partial_1(q_2(iik))q_1(iik).
\end{split}
\end{equation*}
Now it is easy to check that the $\psi_1\psi_2\psi_1$-, $\psi_2\psi_1$-coefficients in the two  expressions above are equal to each other
using (\ref{EP3}) and (\ref{EQProp2}). 
For the $\psi_1\psi_2$-coefficient, we need to use
$$
\partial_1(q_2(iik))+\partial_1({^{s_2}}q_1(iki))
=\partial_1(q_2(iik)+{^{s_1\!}}q_2(iik))=0.
$$
By a calculation using
(\ref{EP}) and (\ref{EPartial}),
$\partial_1(p_2(iik))=-{^{s_1\!}}p_2(iik)p_2(iik)$
and 
$\partial_1({^{s_2}}p_1(iki)) =
{^{s_1\!}}p_2(iik) p_2(iik)$.
Also
$q_1(iik) = 1 + y_2-y_1$  by
(\ref{EQProp1})
hence 
$\partial_1(q_1(iik)) = 2$.
So to check that the $\psi_1$-coefficients agree,
it suffices to prove 
$$
{^{s_1\!}}p_2(iik) p_2(iik) =
p_2(iik)-{^{s_1\!}}p_2(iik) + {^{s_1\!}}p_2(iik)p_2(iik) (1+y_2-y_1).
$$
This follows from the power series identity
\begin{equation}\label{PS2}
p_2(iik) - {^{s_1\!}}p_2(iik) = (y_1-y_2) {^{s_1\!}}p_2(iik)p_2(iik),
\end{equation}
which is easily checked from the definition
(\ref{EP}). 
For the $\psi_2$-coefficients we need to prove
\begin{multline*}
q_2(iik)\left(p_2(iik)-{^{s_2}}p_1(iki)+{^{s_2}}p_2(iki){^{s_2}}p_1(iki)
\right)
=\\q_1(iik)\left(\partial_1(q_2(iik))[p_2(iik)-{^{s_2}}p_1(iki)]
-p_2(iik){^{s_1\!}}p_2(iik){^{s_1\!}}q_2(iik)\right).
\end{multline*} 
Using (\ref{PS2}), (\ref{EP1}) and (\ref{EP3}), the left hand side of this
simplifies to
$$
q_2(iik) (y_1-y_2-1) {^{s_1\!}}p_2(iik) p_2(iik).
$$
Using (\ref{PS2}) again
and expanding the $\partial_1$, 
the right hand side equals
$$
q_1(iik) \left(({^{s_1\!}}q_2(iik) - q_2(iik)){^{s_1\!}}p_2(iik)p_2(iik)
- p_2(iik) {^{s_1\!}}p_2(iik) {^{s_1\!}}q_2(iik)\right).
$$
Making obvious cancellations and recalling (\ref{EQProp1})
this reduces to the left hand side.
Finally, for the constant term we want 
\begin{multline*}
-p_2(iik)^2-{^{s_1\!}}p_2(iik)q_1(iik)p_2(iik)^2
- (\psi_2^2){^{s_2}}q_2(iki){^{s_2}}p_1(iki)q_2(iik)
\\=
-p_2(iik)
-{^{s_1\!}}p_2(iik)p_2(iik)q_1(iik)
+ (\psi_1^2){^{s_1\!}}q_1(iik){^{s_1\!}}p_2(iik)q_1(iik).
\end{multline*}
By (\ref{R4}) the third term on the right is zero.
By (\ref{EUs}) 
the third term on the left equals ${^{s_1\!}}p_2(iik)(1-p_2(iik)^2)$.
Making these substitutions
and replacing $q_1(iik)$ by $(1+y_2-y_1)$, 
the desired equality follows from (\ref{PS2}).

\subsubsection*{Case 3: $i\neq j=k$} This case is similar to Case 2, so we skip it.

\subsubsection*{Case 4: $i=k\neq j$}
Using the equalities 
$p_1(iij)=p_2(jii)=1$, and commuting as usual, (\ref{ELHS}) becomes
\begin{align*}
&-p_2(iji)p_1(iji)p_2(iji)+ \psi_2q_2(iji)p_1(iji)p_2(iji)
+\psi_1q_1(iji)p_2(iji)
\\
&- \psi_2\psi_1{^{s_1\!}}q_2(jii)q_1(iji)p_2(iji)
+\psi_2{^{s_2}} p_2(iij)q_2(iji)
- (\psi_2^2){^{s_2}}q_2(iij)q_2(iji)
\\
&- \psi_1\psi_2{^{s_2s_1\!}}p_2(iij){^{s_2}}q_1(iij)q_2(iji)- \psi_2{^{s_2}}\partial_1(p_2(iij)){^{s_2}}q_1(iij)q_2(iji)
\\
&+ \psi_2\psi_1\psi_2{^{s_2s_1\!}}q_2(iij){^{s_2}}q_1(iij)q_2(iji) + 
(\psi_2^2){^{s_2}}\partial_1(q_2(iij)){^{s_2}}q_1(iij)q_2(iji),
\end{align*}
and (\ref{ERHS}) becomes
\begin{align*}
&-p_1(iji)p_2(iji)p_1(iji)
+ \psi_1q_1(iji)p_2(iji)p_1(iji)+\psi_2q_2(iji)p_1(iji)
\\&- \psi_1\psi_2{^{s_2}}q_1(iij)q_2(iji)p_1(iji)
+ \psi_1{^{s_1\!}}p_1(jii)q_1(iji)
- (\psi_1^2){^{s_1\!}}q_1(jii)q_1(iji)
\\& - \psi_2\psi_1{^{s_1s_2}}p_1(jii){^{s_1\!}}q_2(jii)q_1(iji)- \psi_1{^{s_1}}\partial_2(p_1(jii)){^{s_1\!}}q_2(jii)q_1(iji)
\\& + \psi_1\psi_2\psi_1{^{s_1s_2}}q_1(jii){^{s_1\!}}q_2(jii)q_1(iji)+ (\psi_1^2){^{s_1}}\partial_2(q_1(jii)){^{s_1\!}}q_2(jii)q_1(iji).
\end{align*}
By (\ref{R7}), we have that 
$\psi_1\psi_2\psi_1 - \psi_2\psi_1\psi_2 = \eps$ where
$\eps := 0$ if $i \nslash j$, $1$ if $i \rightarrow j$,
 $-1$ if $i \leftarrow j$ and
$-2y_2+y_1+y_3$ if $i \rightleftarrows j$.
Hence, in view of (\ref{EQProp2}), 
the $\psi_2\psi_1\psi_2$-term cancels with the $\psi_1\psi_2\psi_1$-term, producing the addition of 
$\eps q_1(iji){^{s_2}}q_1(iij)q_2(iji)$
to the constant term of the second 
expression.
Now let us compare constant terms.
Taking into account (\ref{EUs}) and (\ref{EQProp2}), we need to check that 
\begin{multline*}
-p_2(iji)^2p_1(iji)
-(1-p_2(iji)^2)
+ (\psi_2^2){^{s_2}}\partial_1(q_2(iij)){^{s_2}}q_1(iij)q_2(iji)
\\=
-p_1(iji)^2p_2(iji)
- (1-p_1(iji)^2)
+ (\psi_1^2){^{s_1}}\partial_2(q_1(jii)){^{s_1\!}}q_2(jii)q_1(iji)\\
+\eps q_1(iji){^{s_2}}q_1(iij)q_2(iji).
\end{multline*}
Expanding the $\partial$'s, we rewrite this as 
\begin{multline*}
(p_1(iji)-p_2(iji))(p_1(iji)p_2(iji)-p_1(iji)-p_2(iji))\\
+(\psi_2^2)({^{s_2s_1\!}}q_2(iij)-{^{s_2}}q_2(iij)) {^{s_2}}q_1(iij)q_2(iji)
/(y_1-y_3)\\
-(\psi_1^2)({^{s_1s_2}}q_1(jii)-{^{s_1\!}}q_1(jii))
{^{s_1\!}}q_2(jii)q_1(iji)/(y_1-y_3)\\
-\eps q_1(iji){^{s_2}}q_1(iij)q_2(iji)
=0.
\end{multline*}
Using (\ref{PID1}) and (\ref{EQProp1}), the first term on the left hand side equals
$$
(p_1(iji)^2-p_2(iji)^2){^{s_2}}q_1(iji) / (y_1-y_3).
$$
Using (\ref{EUs}) and (\ref{EQProp2}), the second and third terms equal
\begin{align*}
&(q_1(iji) q_2(iji) \psi_2^2-1+p_2(iji)^2){^{s_2}}q_1(iij) / (y_1-y_3)\\\intertext{and}
-&(q_2(iji) q_1(iji) \psi_1^2-1+p_1(iji)^2){^{s_2}}q_1(jii) / (y_1-y_3),
\end{align*}
respectively.
Now we observe by (\ref{R4}) that
\begin{equation}\label{handy}
\psi_2^2 - \psi_1^2 = \eps(y_1-y_3).
\end{equation}
Making these substitutions, 
it is then easy to verify the required identity.

Comparing the $\psi_1\psi_2$-, $\psi_2\psi_1$-, $\psi_1$-, and $\psi_2$-coefficients 
is relatively routine, using (\ref{PID1}) and
the fact that
${^{s_1}}\partial_2(p_1(jii))=
{^{s_2}}\partial_1(p_2(iij))=p_1(iji)p_2(iji)$.

\subsubsection*{Case 5: $i=j=k$} This case is the most explicit of all since we have the precise formulas for all $p$'s and $q$'s. We leave details to the reader.
\end{proof}

\subsection{\boldmath Proof of the Main Theorem in the degenerate case}\label{spf}
Now we can prove the Main Theorem from the introduction when $q=1$.
By Theorem~\ref{dogen}, there is a 
homomorphism 
\begin{equation}\label{sigma}
\RtoH: R_\al^\La\to H_\al^\La
\end{equation}
mapping all the formal generators $e(\bi)$, $y_r$ and $\psi_r$ to the 
explicit elements of $H_\al^\La$ with the same names. 
To prove that $\RtoH$ is an isomorphism, we construct a
two-sided inverse 
\begin{equation}\label{rho}
\HtoR:H_\al^\La\to R_\al^\La.
\end{equation} 
By Theorem~\ref{thm2}, there is a homomorphism
$\HtoR:H_d \to R^\La_\al$ mapping
the formal generators $x_r$ and $s_r$ of $H_d$ to the elements of $R_\al^\La$ 
with the same names defined in (\ref{ET1}) and (\ref{ET2}). 
We claim that this homomorphism 
factors through the natural surjection 
$H_d\twoheadrightarrow H_\al^\La$ to give the required inverse homomorphism 
(\ref{rho}).
To see this, observe 
using the orthogonality of the $e(\bi)$ and (\ref{ERCyc})
that
\begin{align*}
\HtoR\left(\prod_{i\in I}(x_1-i)^{(\La,\al_i)})\right)&
=\sum_{\bj\in I^\al}
\prod_{i\in I}
(y_1+j_1-i)^{(\La,\al_i)}e(\bj)=0.
\end{align*}
Hence, $\HtoR$ factors through the quotient $H^\La_d$ of $H_d$
from (\ref{ECHA}).
To show that $\HtoR$ factors further through the surjection 
$H_d^\La\twoheadrightarrow H_\al^\La$
defined by multiplication by the block idempotent $e_{\al}$,
we need to show that $\HtoR(e_{\be})=0$ for any $\be \in Q_+$
of height $d$ with $\be \neq \al$. 
Recalling (\ref{bal}), this follows immediately from
the following lemma.

\begin{Lemma}\label{l1}
For any $\bi \in I^d$, we have that
$$
\HtoR(e(\bi))=
\left\{
\begin{array}{ll}
e(\bi) &\hbox{if $\bi\in I^\al$};\\
0 &\hbox{otherwise.}
\end{array}
\right.
$$
\end{Lemma}

\begin{proof}
To avoid confusion, let us temporarily denote the idempotent
$e(\bi) \in R^\La_\alpha$ instead by $e(\bi)'$.
Recall that $e(\bi) \in H^\La_d$ is the 
idempotent characterized by the property that 
$e(\bi)M=M_\bi$ for any finite dimensional left $H_d^\La$-module $M$. 
The homomorphism $\HtoR$ makes $R_\al^\La$ into a finite 
dimensional left $H_d^\La$-module. 
By (\ref{ET1}),
Lemma~\ref{LNilY}
and the relation (\ref{R1}),
the weight space $(R_\al^\La)_\bi$ 
is precisely $e(\bi)' R_\al^\La$ 
if $\bi\in I^\al$ and zero otherwise. 
Hence $\HtoR(e(\bi))$ is an idempotent in $R^\La_\al$
that projects $R^\La_\alpha$ onto $e(\bi)' R^\La_\alpha$
if $\bi \in I^\alpha$ and is zero if $\bi \notin I^\alpha$.
Hence $\HtoR(e(\bi)) = e(\bi)'$ if $\bi \in I^\alpha$
and $\HtoR(e(\bi)) = 0$ otherwise.
\end{proof}

This completes the definition of the homomorphism (\ref{rho}).
To finish the proof of the Main Theorem in the degenerate case
it remains to check that
$\RtoH$ and $\HtoR$ are two-sided inverses.
This follows by the final lemma.

\begin{Lemma}\label{l2} 
We have that 
$\HtoR\circ\RtoH= \id_{R_{\al}^\La}$
and 
$\RtoH\circ\HtoR= \id_{H_\al^\La}$.
\end{Lemma}

\begin{proof}
Since both $\HtoR$ and $\RtoH$ are algebra homomorphisms,
it suffices to check that $\HtoR\circ\RtoH$
is the identity on each generator (\ref{kl0})
of $R^\La_\al$ and
that $\RtoH\circ \HtoR$ is the identity on each generator
(\ref{hg0}) of $H^\La_\al$.
The fact that
$\HtoR(\RtoH(e(\bi))) = e(\bi)$ follows from Lemma~\ref{l1}.
Then using (\ref{ET1}) and (\ref{jon}) we get that
$$
\RtoH(\HtoR(x_r)) =
\sum_{\bi \in I^\alpha}
\RtoH((y_r+i_r) e(\bi))
=
\sum_{\bi \in I^\alpha} x_r e(\bi).
$$
Since $\sum_{\bi \in I^\alpha} e(\bi)$
is the identity  $e_{\alpha} \in H^\La_\alpha$, this 
implies $\RtoH(\HtoR(x_r)) = x_r$.
The proof that
$\RtoH(\HtoR(s_r)) = s_r$ is a 
similar calculation using
(\ref{ET2}) and (\ref{ECoxKL}).
Finally one checks that
$\HtoR(\RtoH(y_r)) = y_r$ and
$\HtoR(\RtoH(\psi_r)) = \psi_r$
using the same formulae.
\end{proof}

\section{The non-degenerate case}

\subsection{\boldmath Blocks of   cyclotomic Hecke algebras}\label{QSId}
Now we assume that $q \neq 1$ and 
let $H_d$ be the affine Hecke algebra over $F$
at this parameter.
Thus $H_d$ has generators
\begin{equation}\label{hg1}
\{X_1^{\pm 1}, \dots, X_d^{\pm 1}\}\cup\{T_1,\dots,T_{d-1}\}
\end{equation}
subject to the following relations for all admissible indices:
\begin{align}
\label{QPoly}
X_r^{\pm1} X_s^{\pm1}&=X_s^{\pm1} X_r^{\pm1}, \hspace{7.2mm} X_rX_r^{-1}=1;\\
T_r X_r T_r &= qX_{r+1},
\qquad\qquad
T_r X_s = X_{s}T_r \hspace{6mm}\text{if $s \neq r,r+1$};\label{QDAHA}\\
\label{QECoxeterQuad}
T_r^2 &= (q-1) T_r +q;\\
T_rT_{r+1}T_r&=T_{r+1}T_rT_{r+1},
\qquad
T_rT_s=T_sT_r \hspace{7mm}\text{if $|r-s|>1$}.\label{QCoxeter}
\end{align}
The following relations are easy consequences of the 
defining relations:
\begin{align*}
T_r^{-1}&=q^{-1}T_r+q^{-1}-1;
\\
X_rT_r&=T_rX_{r+1}+(1-q)X_{r+1},\quad X_r^{-1}T_r=T_rX_{r+1}^{-1}+(q-1)X_{r}^{-1};
\\
X_{r+1}T_r&=T_rX_{r}+(q-1)X_{r+1},\hspace{6.5mm} X_{r+1}^{-1}T_r=T_rX_{r}^{-1}+
(1-q)X_{r}^{-1}.
\end{align*}
We'll use these repeatedly without further note.

Given $\La\in P_+$ of level $\ell$ as usual,
denote by the same letters $X_1^{\pm1},\dots,X_d^{\pm1}$ and 
$T_1,\dots,T_{d-1}$ the images of the generators in the cyclotomic 
quotient $H_d^\La$ from (\ref{ECHA}). 
As goes back to \cite{AK}, we have that
\begin{equation}\label{qdim}
\dim H_d^\La =\ell^d d!,
\end{equation} 
just like in the degenerate case.
In case $\ell = 1$, $H_d^\La$ is the usual finite Hecke algebra
associated to the symmetric group $S_d$.

Let $M$ be a finite dimensional $H_d^\La$-module. By 
\cite[Lemma 4.7]{G}, the eigenvalues of each $X_r$ on $M$ are of the 
form $q^i$ for $i\in I$. So 
$M$ decomposes as a direct sum 
$M = \bigoplus_{\bi \in I^d} M_\bi$
of its {\em weight spaces} 
$$
M_\bi := \{v \in M\:|\:(X_r-q^{i_r})^N v = 0
\text{ for all $r=1,\dots,d$ and $N \gg 0$}\}.
$$
As in (\ref{ESpan}), we have that
\begin{equation}\label{QESpan}
T_r(M_\bi)\subseteq M_\bi+M_{s_r\cdot\bi}.
\end{equation}
Considering the weight space decomposition 
of the regular module as in $\S$\ref{SId}, 
we deduce that 
there is a system 
$\{e(\bi)\:|\:\bi \in I^d\}$
of mutually orthogonal idempotents in $H_d^\La$
such that $e(\bi) M = M_\bi$ for each finite dimensional module $M$. 
Each $e(\bi)$
lies in the commutative subalgebra
generated by $X_1^{\pm 1},\dots,X_d^{\pm 1}$,
all but finitely many of the $e(\bi)$'s are zero,
and their sum is the identity element in $H_d^\La$. 

As goes back to Bernstein, 
the center $Z(H_d)$ consists of all symmetric polynomials
in $X_1^{\pm 1},\dots,X_d^{\pm 1}$; see e.g. \cite[Proposition 4.1]{G}. 
So, given also $\alpha \in Q_+$ of height $d$, the 
idempotent $e_{\al}$ from (\ref{bal})
is either zero or it is a central idempotent
in $H_d^\La$. In fact, it follows from \cite{LM} that the non-zero 
$e_{\al}$'s are precisely the primitive central
idempotents of $H_d^\La$ (although we
don't ever use this).
So again the algebra $H_\alpha^\La := e_{\alpha} H_d^\La$ 
from (\ref{fe})
is either zero or it is a block of  $H_d^\La$.
The 
commutative
subalgebra of $H^\La_\al$ generated by $X_1^{\pm1},\dots,X_d^{\pm 1}$ will be denoted $\Pol_\al^\La$. 
Each $\Pol_\al^\La e(\bi)$ is an algebra with identity $e(\bi)$. 
If $X \in \Pol^\La_\al$ is such that $X e(\bi)$ is invertible 
in $\Pol_\al^\La e(\bi)$, we write $X^{-1}e(\bi)$ for its inverse 
in $\Pol_\al^\La e(\bi)$. For example, set 
\begin{equation}
X_{r,s}:=X_rX_s^{-1}. 
\end{equation}
Then $(1-X_{r,s})^{-1}e(\bi)$ makes sense if $i_r\neq i_{s}$. 

\subsection{\boldmath Intertwining elements $\Phi_r$} 
For $1\leq r\leq d$, set
\begin{equation}
\Phi_r:=T_r+\sum_{\substack{\bi\in I^\al\\i_r\neq i_{r+1}}}
(1-q)(1-X_{r,r+1})^{-1}e(\bi)
+\sum_{\substack{\bi\in I^\al\\i_r= i_{r+1}}}e(\bi).
\end{equation}
This is a slightly modified version of the usual intertwining elements
as in \cite[\S2]{Ro} and \cite[\S5.1]{Lu}:
\begin{equation}\label{QERelation}
\Theta_r:=T_r(1-X_{r,r+1})+1-q.
\end{equation}
The elements $\Theta_r$ have the following nice properties 
which are checked from the relations; see
\cite[Proposition 5.2]{Lu}:
\begin{align}
\label{QEOldI4}
\Theta_r^2&=
(1-qX_{r+1,r})(1-qX_{r,r+1});\\
\label{QEOldI1}
\Theta_rX_{r+1}^{\pm1}&=
X_r^{\pm1}\Theta_r,\hspace{17.5mm}
\Theta_rX_s=X_s\Theta_r\hspace{6mm}\text{if $s\neq r,r+1$};
\\
\label{QEOldI6}
\Theta_r\Theta_{r+1}\Theta_r&=\Theta_{r+1}\Theta_r\Theta_{r+1},
\hspace{8.3mm} 
\Theta_r\Theta_s=\Theta_s\Theta_r\hspace{6.2mm}\text{if $|r-s|>1$}.
\end{align}
The elements $\Phi_r$ inherit similar properties:

\begin{Lemma}
The intertwining elements satisfy the following relations for all $\bi \in I^\alpha$ and admissible $r,s$:
\begin{align}
\label{QEI0}
\Phi_re(\bi)&=e(s_r{\cdot}\bi)\Phi_r;\\
\label{QEI1}
\Phi_rX_s&=X_s\Phi_r\hspace{52.8mm}\text{if $s\neq r,r+1$};\\
\label{QEI5}
\hspace{12.1mm}\Phi_r\Phi_s&=\Phi_s\Phi_r\hspace{53.1mm}\text{if $|r-s|>1$};\\
\label{QEI2}
\Phi_rX_{r+1}e(\bi)&
=\begin{cases}
X_r\Phi_re(\bi) & \text{if $i_r\neq i_{r+1}$},\\
(X_r\Phi_r+qX_{r+1}-X_r)e(\bi) \hspace{11.6mm}&\text{if $i_r= i_{r+1}$};
\end{cases}\\
\label{QEI3}
X_{r+1}\Phi_re(\bi)&=\begin{cases}
\Phi_rX_re(\bi) & \text{if $i_r\neq i_{r+1}$},\\
(\Phi_rX_r+qX_{r+1}-X_r)e(\bi)\hspace{11.6mm} &\text{if $i_r= i_{r+1}$};
\end{cases}\\
\label{QEI4}
\Phi_r^2e(\bi)&=\begin{cases}
\frac{(X_{r+1}-qX_r)(X_r-qX_{r+1})}{(X_{r+1}-X_{r})(X_{r}-X_{r+1})}e(\bi)\hspace{13.2mm} & \text{if $i_r\neq i_{r+1}$},\\
(1+q)\Phi_re(\bi) &\text{if $i_r= i_{r+1}$};
\end{cases}\end{align}\begin{align}
\label{QEI6}
\:\:\Phi_r\Phi_{r+1}\Phi_r
&=\begin{cases}
(\Phi_{r+1}\Phi_r\Phi_{r+1}+q\Phi_{r}-q\Phi_{r+1})e(\bi) \!\!& \text{if $i_r=i_{r+2}=i_{r+1}$},\\
(\Phi_{r+1}\Phi_r\Phi_{r+1}+Z_r
)e(\bi) & \text{if $i_r=i_{r+2}\neq i_{r+1}$},\\
\Phi_{r+1}\Phi_r\Phi_{r+1}e(\bi) & \text{otherwise},
\end{cases}
\end{align}
where $Z_r$ denotes $(1-q)^2\frac{(X_rX_{r+2}-X_{r+1}^2)(X_r X_{r+1}-qX_{r+1}X_{r+2})}{(X_r-X_{r+1})^2(X_{r+1}-X_{r+2})^2}$.
\end{Lemma}

\begin{proof}
The first equation (\ref{QEI0}) follows by (\ref{QESpan})
and (\ref{QEOldI1})
using the fact that $\Phi_r e(\bi) = \Theta_r (1-X_{r,r+1})^{-1} e(\bi)$
for $i_r \neq i_{r+1}$, in the same way that 
(\ref{EI0}) was verified in the proof of Lemma~\ref{int}.
The properties (\ref{QEI1}) and (\ref{QEI5}) are clear from definitions. 
The properties (\ref{QEI2}) and 
(\ref{QEI3}) come easily from (\ref{QEOldI1}) and relations in $H_\al^\La$. 
For (\ref{QEI4}), if $i_r=i_{r+1}$, we have
$\Phi_r^2e(\bi)=(T_r+1)^2e(\bi)=(1+q)\Phi_re(\bi)$.
Now suppose that 
$i_r\neq i_{r+1}$. 
Then, using (\ref{QEOldI4}) and (\ref{QEOldI1}), we have that
\begin{align*}
\Phi_r^2e(\bi)
=&\Theta_r(1-X_{r,r+1})^{-1}\Theta_r(1-X_{r,r+1})^{-1}e(\bi)
\\
=&\Theta_r^2(1-X_{r+1,r})^{-1}(1-X_{r,r+1})^{-1}e(\bi)
\\
=&(1-qX_{r+1,r})(1-qX_{r,r+1})(1-X_{r+1,r})^{-1}(1-X_{r,r+1})^{-1}e(\bi).
\end{align*}
Finally, for the proof of (\ref{QEI6}), we
assume without loss of generality that $r=1$ and  consider five cases
like in the proof of Lemma~\ref{int}.
Case 4 involves making a lengthy but routine
expansion.
\end{proof}

\subsection{\boldmath Khovanov-Lauda generators of $H^\La_\al$
in the non-degenerate case}\label{QSKLGens}
Set 
\begin{equation}\label{QEPolKL}
y_r:=\sum_{\bi\in I^\al}(1 - q^{-i_r}X_r)e(\bi).
\end{equation}
Note that $y_1,\dots,y_d\in \Pol_\al^\La$ are nilpotent, so we are in
the situation of (\ref{conv})
and can interpret any power series in $F[[y_1,\dots,y_d]]$
as an element of $\Pol^\La_\al$.
It is convenient also to set
\begin{equation}\label{X}
y_r(\bi) := q^{i_r}(1-y_r) \in F[[y_1,\dots,y_d]],
\end{equation}
so that 
\begin{equation}\label{Xy}
X_r e(\bi) = y_r(\bi) e(\bi)
\end{equation} 
for each $\bi \in I^\alpha$.
We note further that
\begin{align}\label{jona}
y_{r+1}-q y_r(\bi) &= q^{i_{r+1}}(y_r-y_{r+1})\qquad\text{if $i_r \rightarrow i_{r+1}$ or $i_r \rightleftarrows i_{r+1}$},\\
y_r(\bi)-q y_{r+1}(\bi) &= q^{i_{r}}(y_{r+1}-y_{r})\qquad\!\quad\text{if $i_r \leftarrow i_{r+1}$ or $i_r \rightleftarrows i_{r+1}$},\label{jonb}\\
{^{s_r\!}}y_r(s_r{\cdot}\bi)
&=
y_{r+1}(\bi)
\hspace{22mm}\text{for all $\bi$.}\label{jonc}
\end{align}
For any $1\leq r<d$ and $\bi\in I^\al$
we define power series
$P_r(\bi) \in F[[y_r,y_{r+1}]]$ by setting
\begin{equation}
\label{QEP}
P_r(\bi):=\begin{cases}
1 & \text{if $i_r=i_{r+1}$},\\
(1-q)\left(1-y_r(\bi) y_{r+1}(\bi)^{-1}\right)^{-1} & \text{if $i_r\neq i_{r+1}$}.
\end{cases}
\end{equation}
The following facts are easy to check:
\begin{align}
P_r(\bi)+{^{s_r\!}}P_r(s_r{\cdot}\bi) &= 1-q\hspace{40.3mm}\text{if $i_r \neq i_{r+1}$};\label{qr}\\
{^{s_r\!}}P_{r+1}(s_r{\cdot}\bi)&={^{s_{r+1}\!}}P_r(s_{r+1}{\cdot}\bi)
\hspace{26mm}\text{for any $\bi$};\label{QEP3}
\\
\label{QEPE}
P_r(\bi)e(\bi)&=\begin{cases}
e(\bi) & \text{if $i_r=i_{r+1}$},\\
(1-q)(1-X_{r,r+1})^{-1}e(\bi) & \text{if $i_r\neq i_{r+1}$};
\end{cases}
\\
\label{QEPPhi}
\Phi_r&=\sum_{\bi\in I^\al}(T_r+P_r(\bi))e(\bi).  
\end{align}
By explicitly expanding both sides in terms of $y_1,\dots,y_d$, one checks that
\begin{equation}\label{QEPreUs}
(1-P_r(\bi))(q+P_r(\bi))
=\frac{(y_{r+1}(\bi)-qy_r(\bi))
(y_r(\bi)-qy_{r+1}(\bi))}{(y_{r+1}(\bi)-y_{r}(\bi))(y_{r}(\bi)-y_{r+1}(\bi))}
\end{equation}
for all $\bi \in I^\alpha$ with $i_r \neq i_{r+1}$.
Note the denominator on the right hand side of (\ref{QEPreUs})
is a unit in $F[[y_r,y_{r+1}]]$ so this makes sense.

Fix from now on a choice of invertible elements
$Q_r(\bi)\in F[[y_r,y_{r+1}]]$
with the following properties:
\begin{align}
\label{QEQProp1}
Q_r(\bi)&=1-q+qy_{r+1}-y_r\hspace{14.5mm}\text{if $i_r=i_{r+1}$};\\
\label{QEQProp2}
Q_r(\bi){^{s_r\!}}Q_r(s_r{\cdot}\bi)&=
\left\{
\begin{array}{ll}
(1-P_r(\bi))(q+P_r(\bi))&\text{if $i_r \nslash i_{r+1}$},\\
\frac{(1-P_r(\bi))(q+P_r(\bi)))}{y_{r+1}-y_r}
&\hbox{if $i_r\rightarrow i_{r+1}$},
\\
\frac{(1-P_r(\bi))(q+P_r(\bi))}{y_r - y_{r+1}}
 &\hbox{if $i_r\leftarrow i_{r+1}$},
\\
\frac{(1-P_r(\bi))(q+P_r(\bi))}{(y_{r+1}-y_r)(y_r-y_{r+1})}&\hbox{if $i_r\rightleftarrows i_{r+1}$};
\end{array}
\right.\\
\label{QEQProp4}
{^{s_r}}Q_{r+1}(s_{r+1}s_r{\cdot}\bi)&={^{s_{r+1}}}Q_r(s_rs_{r+1}{\cdot}\bi) \hspace{18mm}
\text{for any $\bi$}.
\end{align}
In the fractions on the right hand side of (\ref{QEQProp2}),
the fact that the denominators divide the numerators follows because
of (\ref{QEPreUs}) and (\ref{jona})--(\ref{jonb}).
For example, one could choose
\begin{equation}
Q_r(\bi):=
\left\{
\begin{array}{ll}
1-q+qy_{r+1}-y_r&\text{if $i_r = i_{r+1}$},\\
(y_r(\bi)-qy_{r+1}(\bi)))/(y_{r}(\bi)-y_{r+1}(\bi)) &\hbox{if $i_r \nslash i_{r+1}$},\\
(y_r(\bi)-qy_{r+1}(\bi))/(y_r(\bi)-y_{r+1}(\bi))^2 &\hbox{if $i_r\rightarrow i_{r+1}$},\\
q^{i_r} &\hbox{if $i_r\leftarrow i_{r+1}$},\\
q^{i_r}/(y_{r}(\bi)-y_{r+1}(\bi)) &\hbox{if $i_r \rightleftarrows i_{r+1}$},
\end{array}
\right.
\end{equation}
which satisfy (\ref{QEQProp2}) by another application of (\ref{QEPreUs})
and (\ref{jona})--(\ref{jonc}).

Now, the {\em Khovanov-Lauda generators} in the non-degenerate case are
\begin{equation}\label{Qklg}
\{e(\bi)\mid\bi\in I^\al\}\cup\{y_1,\dots,y_d\}\cup\{\psi_1,\dots,\psi_{d-1}\},
\end{equation}
where $y_r$ is the element defined by (\ref{QEPolKL}) and 
\begin{equation}\label{QQCoxKL}
\psi_r:=\sum_{\bi\in I^\al}\Phi_rQ_r(\bi)^{-1}e(\bi) 
=\sum_{\bi\in I^\al}(T_r+P_r(\bi))Q_r(\bi)^{-1}e(\bi).
\end{equation}

\begin{Theorem}\label{Qgen0}
The elements (\ref{Qklg})
of $H^\La_\alpha$ satisfy the defining relations (\ref{ERCyc})--(\ref{R7})
of the cyclotomic Khovanov-Lauda algebra.
\end{Theorem}

\begin{proof}
The checks of (\ref{ERCyc})--(\ref{R3Psi})
go in the same way as in the proof of Theorem~\ref{dogen}; 
these are easy so we omit them.
For (\ref{R5}), we have that
\begin{equation}\label{ESunday2}
y_{r+1} \psi_r e(\bi)=(1-q^{-i_r} 
X_{r+1})\Phi_rQ_r(\bi)^{-1}e(\bi).
\end{equation}
If $i_r\neq i_{r+1}$, 
this equals  
$\Phi_rQ_r(\bi)^{-1}(1-q^{-i_r} X_{r})e(\bi)=\psi_ry_re(\bi)$ by (\ref{QEI3}).
If $i_r=i_{r+1}$, then (\ref{ESunday2}) gives 
\begin{align*}
y_{r+1} \psi_r e(\bi) &=(1-q^{-i_r}X_{r+1})
(T_r+1) Q_r(\bi)^{-1}e(\bi)\\
&=((T_r+1)(1-q^{-i_r}X_r)+q^{-i_r}X_r-q^{1-i_r}X_{r+1})
Q_r(\bi)^{-1}
e(\bi)\\&=(\psi_ry_r+1)e(\bi),
\end{align*}
since $(q^{-i_r} X_r - q^{1-i_r} X_{r+1})e(\bi)=Q_r(\bi)e(\bi)$.
The proof of (\ref{R6}) is similar. 

As we have now verified 
(\ref{R3YPsi}), (\ref{R6}) and (\ref{R5}), we can make use of the
identity (\ref{EDDR}) in $H^\La_\alpha$ when necessary.
For (\ref{R4}), using (\ref{R2PsiE}), we have:
\begin{equation}\label{QESat}
\psi_r^2e(\bi) =\Phi_rQ_r(s_r{\cdot}\bi)^{-1}\psi_re(\bi). 
\end{equation}
If $i_r\neq i_{r+1}$, then by (\ref{EDDR}), 
this becomes
\begin{align*}
\Phi_r\psi_r{^{s_r}}Q_r(s_r{\cdot}\bi)^{-1}e(\bi)&=\Phi_r^2(Q_r(\bi){^{s_r}}
Q_r(s_r{\cdot}\bi))^{-1}e(\bi).
\end{align*}
By (\ref{QEI4}), (\ref{Xy}) and (\ref{QEPreUs}),
we have that $\Phi_r^2 e(\bi)  = (1-P_r(\bi))(q+P_r(\bi))e(\bi)$,
hence this expression simplifies to give 
the right hand side of (\ref{R4}) by (\ref{QEQProp2}).
Now, let $i_r=i_{r+1}$. Then, using (\ref{EDDR}) again, (\ref{QESat}) becomes
\begin{align*}
\psi_r^2 e(\bi)&=
(T_r+1)(1-q+qy_{r+1}-y_r)^{-1}\psi_re(\bi)
\\&=(T_r+1)
(T_r-q)(1-q+qy_r-y_{r+1})^{-1} (1-q+qy_{r+1}-y_r)^{-1}
e(\bi),
\end{align*}
which is zero by (\ref{QECoxeterQuad}).

Finally we prove (\ref{R7}). 
Let us also stop writing $e(\bi)$ on the right of 
all formulas.  
Assume without loss of generality that $r=1$, $d=3$,  denote $i := i_1, j := 
i_2, k := i_3$, and consider the usual five cases.

\subsubsection*{Case 1: $i,j,k$ all distinct}
This is exactly the same calculation as Case 1 
from the proof of Theorem~\ref{dogen};
one needs to use (\ref{EDDR}), 
(\ref{QEI6}) and (\ref{QEQProp4}).

\subsubsection*{Case 2: $i=j\neq k$}
This is exactly the same calculation as Case 2 
from the proof of Theorem~\ref{dogen}.

\subsubsection*{Case 3: $i\neq j=k$} Similar.

\subsubsection*{Case 4: $i=k\neq j$} Expanding as 
in Case 4 from the proof of Theorem~\ref{dogen} and using (\ref{QEI4}),
(\ref{QEI6}), (\ref{QEQProp1}) and (\ref{QEQProp4}), we get
that
$\psi_1\psi_2\psi_1-\psi_2\psi_1\psi_2 = A+B-C$
where
\begin{align*}
A&=(1-q)^2\frac{(X_1X_3-X_2^2)(X_1X_2-qX_2X_3)}{(X_1-X_2)^2(X_2-X_3)^2}
\left(Q_1(iji){^{s_1}}Q_2(jii)Q_2(iji)\right)^{-1},\\
B&=\frac{(X_2-qX_1)(X_1-qX_2)}{(X_2-X_1)(X_1-X_2)}
Q_1(iji)^{-1}\left({^{s_1}}\partial_2(Q_1(jii)^{-1})\right),\\
C&=\frac{(X_3-qX_2)(X_2-qX_3)}{(X_3-X_2)(X_2-X_3)}
Q_2(iji)^{-1}\left({^{s_2}}\partial_1(Q_2(iij)^{-1})\right).
\end{align*}
Noting that ${^{s_1}}Q_2(jii) = q^{-i} (X_1-qX_3)$,
we get that
$$
A = q^i(1-q)^2 \frac{(X_1X_3-X_2^2)X_2}{(X_1-X_2)^2(X_2-X_3)^2}Q_1(iji)^{-1} Q_2(iji)^{-1}.
$$
Expanding the $\partial$'s and using (\ref{QEQProp4}), we also have that
\begin{align*}
B &= 
\frac{(X_2-qX_1)(X_1-qX_2)}{(X_2-X_1)(X_1-X_2)}
\cdot\frac{Q_1(iji)^{-1}Q_2(iji)^{-1}-(Q_1(iji){^{s_1}}Q_1(jii))^{-1}}{y_1-y_3},\\
C &= 
\frac{(X_3-qX_2)(X_2-qX_3)}{(X_3-X_2)(X_2-X_3)}
\cdot\frac{Q_1(iji)^{-1}Q_2(iji)^{-1}-(Q_2(iji){^{s_2}}Q_2(iij))^{-1}}{y_1-y_3}.
\end{align*}
Note also that $y_1-y_3 = -q^{-i} (X_1-X_3)$, and by a direct expansion we have the identity
\begin{multline*}
(1-q)^2 \frac{(X_1X_3-X_2^2)X_2}{(X_1-X_2)^2
(X_2-X_3)^2}
+ \frac{(X_2-qX_1)(X_1-qX_2)}{(X_1-X_2)^2(X_1-X_3)}
\\-\frac{(X_3-qX_2)(X_2-qX_3)}{(X_2-X_3)^2(X_1-X_3)}=0.
\end{multline*}
Combining these things gives that
\begin{multline*}
A+B-C = 
-\frac{(X_2-qX_1)(X_1-qX_2)}{(X_2-X_1)(X_1-X_2)}
\cdot\frac{(Q_1(iji){^{s_1}}Q_1(s_1{\cdot}iji))^{-1}}{y_1-y_3}\\+
\frac{(X_3-qX_2)(X_2-qX_3)}{(X_3-X_2)(X_2-X_3)}
\cdot \frac{(Q_2(iji){^{s_2}}Q_2(s_2{\cdot}iji))^{-1}}{y_1-y_3}.
\end{multline*}
Now use (\ref{Xy}), (\ref{QEPreUs}) and (\ref{QEQProp2}) to deduce that this
equals
$0$
if $i \nslash j$, 
$1$ if $i \rightarrow j$, $-1$ if $i \leftarrow j$
or $-2y_2+y_1+y_3$ if $i \rightleftarrows j$.
This imples (\ref{R7}).

\subsubsection*{Case 5: $i=j=k$} Exercise.
\end{proof} 

\subsection{\boldmath Hecke generators of $R^\La_\al$}\label{QSKLR}
Once again, we let $R^\La_\alpha$ be the cyclotomic Khovanov-Lauda algebra
from $\S$\ref{sckl}.
Using the homomorphism (\ref{Rhom}),
we can regard the power series
$P_r(\bi)$, $Q_r(\bi)$ and $y_r(\bi)$ from $\S$\ref{QSKLGens}
as elements of $R_\al^\La$.
The {\em Hecke generators} of $R_\al^\La$ are the elements
\begin{equation}\label{Qhg}
\{X_1,\dots,X_d\}\cup\{T_1,\dots,T_{d-1}\}
\end{equation}
where
\begin{align}
 X_r&:=\sum_{\bi\in I^\al}
y_r(\bi) e(\bi) =
\sum_{\bi \in I^\al} q^{i_r}(1-y_r) e(\bi),
 \label{QET1}
 \\
 T_r&:=\sum_{\bi\in I^\al}
 (\psi_rQ_r(\bi)-P_r(\bi))e(\bi). 
 \label{QET2}
\end{align}
We often need the following 
consequence of (\ref{R4}) and (\ref{QEQProp2}):
\begin{equation}\label{QEUs}
\psi_r^2Q_r(\bi){^{s_r}}Q_r(s_r{\cdot}\bi)e(\bi)
=
(1-P_r(\bi))(q+P_r(\bi))e(\bi).
\end{equation}
Now we are ready to make the final set of computations.

\begin{Theorem}\label{Qthm2}
The elements (\ref{Qhg})
of $R^\La_\alpha$ 
satisfy the defining relations (\ref{QPoly})--(\ref{QCoxeter})
of the affine Hecke algebra $H_d$.
\end{Theorem}

\begin{proof}
The relations (\ref{QPoly}), (\ref{QDAHA}) for $s \neq r,r+1$,
and (\ref{QCoxeter}) for $|r-s|>1$ are obvious.
Next we check (\ref{QECoxeterQuad}), i.e. $T_r^2 e(\bi) = (q-1)T_r e(\bi)
+ q e(\bi)$
for all $\bi \in I^\alpha$.
By exactly the same calculation as in the proof of Theorem~\ref{thm2}, 
we have that
$$
T_r^2 = (\psi_r Q_r(s_r{\cdot}\bi)\psi_r Q_r(\bi)
-P_r(s_r{\cdot}\bi)\psi_rQ_r(\bi) - \psi_r Q_r(\bi) P_r(\bi)
+ P_r(\bi)^2)e(\bi).
$$
If $i_r  =i_{r+1}$ we use 
the facts $\psi_r^2 e(\bi) = 0$ and
$\partial_r(1-q+qy_{r+1}-y_r) = q+1$ combined with
(\ref{EDDR})
to deduce that
\begin{align*}
T_r^2 e(\bi) &= (\psi_r (1-q+qy_{r+1}-y_r) \psi_r Q_r(\bi)
- 2 \psi_r Q_r(\bi) + 1) e(\bi)\\
&=
((q+1) \psi_r Q_r(\bi) - 2 \psi_r Q_r(\bi) +1) e(\bi)\\
&=
((q-1) (\psi_r Q_r(\bi)-1) + q)e(\bi) = ((q-1) T_r + q) e(\bi).
\end{align*}
If $i=j$ then by (\ref{EDDR}), (\ref{QEUs}) and (\ref{qr}) we get instead that
\begin{align*}
T_r^2 &= ((\psi_r^2) {^{s_r}}Q_r(s_r{\cdot}\bi) Q_r(\bi)-\psi_r(P_r(\bi)+{^{s_r\!}}P_r(s_r{\cdot}\bi))Q_r(\bi)+P_r(\bi)^2)e(\bi)\\
&=((1-P_r(\bi))(q+P_r(\bi)) + (q-1)\psi_rQ_r(\bi)+P_r(\bi)^2 )e(\bi)\\
&= ((q-1)(\psi_r Q_r(\bi)-P_r(\bi))+q) e(\bi) = ((q-1) T_r + q)e(\bi).
\end{align*}
This completes the proof of the quadratic relation.

Next consider the remaining mixed relation from (\ref{QDAHA}).
As we have already checked the quadratic relation,
it suffices to show that $X_r T_r e(\bi) = (T_r +1-q) X_{r+1}e(\bi)$
for each $\bi$.
Using (\ref{EDDR}), (\ref{jonc}) 
and the fact that $\partial_r(y_r(\bi)) = q^{i_r}$, 
we get that
\begin{align*}
X_r T_r e(\bi) &= 
(y_r (s_r {\cdot} \bi) \psi_r Q_r(\bi) - y_r(\bi) P_r(\bi))e(\bi)\\
&=(\psi_r y_{r+1}(\bi) Q_r(\bi) + \delta_{i_r, i_{r+1}} q^{i_r} Q_r(\bi)
- y_r(\bi) P_r(\bi))e(\bi),\\
(T_r+1-q)X_{r+1}e(\bi) &=
(\psi_r Q_r(\bi) - P_r(\bi) + 1-q) y_{r+1}(\bi) e(\bi).
\end{align*}
Considering the two cases $i_r \neq i_{r+1}$ and $i_r = i_{r+1}$
separately,
it's now an easy exercise to check these two expressions are equal
using (\ref{X}), (\ref{QEP}) and (\ref{QEQProp1}).

This just leaves the braid relations.
We assume that $r=1$, $d=3$, set
$i := i_1, j := j_1, k := k_1$, and need to show that
$T_2 T_1 T_2 e(ijk) = T_1T_2T_1 e(ijk)$.
As usual we stop writing $e(ijk)$ on the right of all expressions.
To start with, the left (resp.\ right) hand side of the identity to be checked
expands to exactly the same expression as (\ref{ELHS})
(resp.\ (\ref{ERHS})),
with all $p_r$ and $q_r$ there replaced by $P_r$ and $Q_r$.
Then we consider the usual five cases.

\subsubsection*{Case 1: $i,j,k$ are all different}
This is entirely similar to Case 1 in the proof of Theorem~\ref{thm2}. 
When equating $\psi_1$-, $\psi_2$- and constant coefficients at the end,
everything reduces 
to the following three identities:
\begin{align*}
{^{s_1\!}}P_2(jik)P_2(ijk)
&=P_1(ijk)P_2(ijk)+{^{s_1\!}}P_1(jik){^{s_1\!}}P_2(jik),\\
{^{s_2}}P_1(ikj)P_1(ijk)
&=P_1(ijk)P_2(ijk)+{^{s_2}}P_1(ikj){^{s_2}}P_2(ikj),\\
P_1(ijk)P_2(ijk)^2+ &(\psi_2^2){^{s_2}}Q_2(ikj){^{s_2}}P_1(ikj)Q_2(ijk)\\
&=P_1(ijk)^2P_2(ijk)+(\psi_1^2){^{s_1}}Q_1(jik){^{s_1\!}}P_2(jik)Q_1(ijk).
\end{align*}
The first two of these are easily checked from
(\ref{QEP}) and (\ref{jonc}).
The last one follows using (\ref{QEUs}) too.

\subsubsection*{Case 2: $i=j\neq k$} 
Again, the initial expansions made in Case 2 of Theorem~\ref{thm2}
are valid in the present situation.
We need to equate coefficients on both sides of these
expansions. As before,
this is easy until we get to the $\psi_1$-, $\psi_2$-
and constant coefficients.
For the $\psi_1$-coefficients, noting that
$\partial_1(Q_1(iik))= 1+q$, we have to show that
\begin{equation}\label{night}
\partial_1(P_2(iik))Q_1(iik) = 
P_2(iik)-q\,{^{s_1\!}}P_2(iik) - P_2(iik){^{s_1\!}}P_2(iik).
\end{equation}
This can be checked by brute force, expanding both sides fully using
(\ref{QEP}), (\ref{X}) and (\ref{QEQProp1}).
Next, for the $\psi_2$-coefficients, we need to show that
\begin{multline*}
(P_2(iik)-{^{s_2}}P_1(iki)+{^{s_2}}P_1(iki){^{s_2}}P_2(iki))Q_2(iik)
=\\
\left[
\partial_1(Q_2(iik))(P_2(iik)-{^{s_2}}P_1(iki))-{^{s_1}}Q_2(iik) \partial_1({^{s_2}} P_1(iki))
\right]Q_1(iik).
\end{multline*}
Replacing ${^{s_2}}P_1(iki)$ with ${^{s_1\!}}P_2(iik)$
everywhere and using (\ref{qr}) to rewrite the term
${^{s_2}}P_2(iki)$ as $1-q-P_2(iik)$, this identity is equivalent to
\begin{multline}
\!\!\!\!(P_2(iik)-q\,{^{s_1\!}}P_2(iik)-P_2(iik){^{s_1\!}}P_2(iik))Q_2(iik)
=
\\
\:\left[
\partial_1(Q_2(iik))(P_2(iik)-{^{s_1\!}}P_2(iik))-{^{s_1}}Q_2(iik) \partial_1({^{s_1\!}} P_2(iik))
\right]Q_1(iik).\!\!\label{bed}
\end{multline}
Now we expand the $\partial_1$'s on the right hand side of (\ref{bed}) to
see that it equals
$$
\frac{{^{s_1\!}}P_2(iik)-P_2(iik)}{y_1-y_2}
Q_1(iik)Q_2(iik).
$$
This equals $\partial_1(P_2(iik)) Q_1(iik) Q_2(iik)$
which by (\ref{night}) is equal to the left hand side of 
(\ref{bed}).
Finally, to check the constant coefficients, we need to show that
\begin{multline*}
P_2(iik)(1-P_2(iik))- 
(1-P_2(iik))(q+P_2(iik)){^{s_2}}P_1(iki)\\
=
\partial_1(P_2(iik))Q_1(iik)(1-P_2(iik)),
\end{multline*}
where we have used (\ref{QEUs}) and the
observation that $\psi_1^2 = 0$ by (\ref{R4}).
Noting that ${^{s_2}}P_1(iki) = {^{s_1\!}}P_2(iik)$,
this follows from (\ref{night}) on multiplying both sides
by $(1-P_2(iik))$.

\subsubsection*{Case 3: $i\neq j=k$} Similar.

\subsubsection*{Case 4: $i=k\neq j$}
Again we follow the calculation from Case 4 of
Theorem~\ref{thm2}. Let 
$\eps := 0$ if $i \nslash j$, $1$ if $i \rightarrow j$,
 $-1$ if $i \leftarrow j$ and
$-2y_2+y_1+y_3$ if $i \rightleftarrows j$.
In view of (\ref{EQProp2}), 
the $\psi_2\psi_1\psi_2$-term cancels with the $\psi_1\psi_2\psi_1$-term, producing the addition of 
$\eps q_1(iji){^{s_2}}q_1(iij)q_2(iji)$
to the constant term of the second 
expression.
Now we just consider the constant coefficients,
all the other coefficients being routine.
Using (\ref{QEUs}), we need to show that
\begin{multline*}
-P_2(iji)^2P_1(iji)
- (1-P_2(iji))(q+P_2(iji))\\
+ (\psi_2^2) {^{s_2}}\partial_1(Q_2(iij)){^{s_2}}Q_1(iij) Q_2(iji)
\end{multline*}
equals
\begin{multline*}
-P_1(iji)^2P_2(iji) - (1-P_1(iji))(q+P_1(iji)) + \\(\psi_1^2) 
{^{s_1\!}} \partial_2(Q_1(jii)){^{s_1}}Q_2(jii) Q_1(iji)
+ \eps Q_1(iji){^{s_2}}Q_1(iij) Q_2(iji).
\end{multline*}
Expanding the $\partial$'s, this means we have to show that
\begin{multline*}
P_1(iji)P_2(iji)(P_1(iji)-P_2(iji))\\
+ (1-P_1(iji))(q+P_1(iji))
- (1-P_2(iji))(q+P_2(iji))\\
+
(\psi_2^2) \frac{Q_1(iji)Q_2(iji)-{^{s_2}}Q_2(iij)Q_2(iji)}{y_1-y_3} {^{s_2}}Q_1(iij) \\
-
(\psi_1^2) \frac{Q_1(iji)Q_2(iji)-Q_1(iji){^{s_1}}Q_1(jii)}{y_1-y_3} {^{s_2}}Q_1(iij)\\
-\eps Q_1(iji)Q_2(iji){^{s_2}}Q_1(iij)=0.
\end{multline*}
Now we use (\ref{handy}) and (\ref{QEUs}) and simplify 
to get that the
left hand side of this expression equals
\begin{multline*}
\!((1-P_1(iji))(q+P_1(iji))-(1-P_2(iji))(q+P_2(iji)) \frac{{^{s_2}}Q_1(iij)+y_1-y_3}{y_1-y_3}\\+P_1(iji)P_2(iji)(P_1(iji)-P_2(iji)).
\end{multline*}
Simplifying the first term using (\ref{QEQProp1}) then cancelling
$(P_1(iji)-P_2(iji))$ everywhere, this reduces to checking
$$
(1-q)(1-q-P_1(iji)-P_2(iji))(1-y_3)+P_1(iji)P_2(iji)(y_1-y_3)=0,
$$
which is straightforward using (\ref{QEP}) then (\ref{X}).

\subsubsection*{Case 5: $i=j=k$} Exercise.
\end{proof}

\subsection{\boldmath Proof of the Main Theorem in the non-degenerate case}
Finally we can prove the Main Theorem from the introduction for $q \neq 1$.
By Theorem~\ref{Qgen0}, there is a 
homomorphism 
\begin{equation}\label{Qsigma}
\RtoH: R_\al^\La\to H_\al^\La
\end{equation}
mapping all the formal generators $e(\bi)$, $y_r$ and $\psi_r$ to the 
explicit elements of $H_\al^\La$ with the same names. 
This homomorphism is an isomorphism because it has a two-sided inverse
\begin{equation}\label{Qrho}
\HtoR:H_\al^\La\to R_\al^\La
\end{equation} 
mapping the generators $T_r$ and $X_r$ to the explicit
elements of $R^\La_\al$ with the same names.
This statement follows
by Theorem~\ref{Qthm2} and arguments that are
entirely similar to those in $\S$\ref{spf}.

\section{Example: Young's semi-normal form}\label{SSEx}

In this section we make a rather drastic special assumption: 
$\La=\La_0$ and 
$e=0$. With this assumption, the degenerate cyclotomic Hecke algebra 
$H^\La_d$ is equal to 
the 
group algebra $FS_d$ of the symmetric 
group if $q = 1$ or the corresponding
Iwahori-Hecke algebra if $q \neq 1$.
Either way, $H^\La_d$ is a semisimple algebra, so its
blocks are matrix algebras.
We want to explain how in this special case our results basically reduce to 
the classical Young's semi-normal form. From this point of view, one can think 
loosely of our Main Theorem as a 
replacement for Young's semi-normal form when the
blocks are not simple.

Let $\la$ be a partition of $d$. We identify $\la$ with its Young diagram 
drawn in the usual English notation. 
The {\em residue} of the box of $\la$ in row $i$ and column $j$ 
is defined to be $j-i \in I$. 
By a {\em $\la$-tableau} we mean a diagram obtained by filling the
boxes of $\la$ with the entries $1,\dots,d$ (each appearing exactly once).
Let $\Tab(\la)$ denote
the usual set of all {\em standard $\la$-tableaux},
i.e. the $\la$-tableaux
whose entries are strictly increasing both along rows from left to right
and down columns from top to bottom.
The symmetric group $S_d$ acts naturally on the set of all
$\la$-tableaux by its action on the entries,
but it 
does not preserve the subset $\Tab(\la)$ of 
standard $\la$-tableaux. 

To any $\T\in\Tab(\la)$ 
we associate its {\em residue sequence} $\bi^{\T} := 
(i_1,\dots,i_d)$, where $i_m$ is the residue of the box of
$\T$ containing the entry
$m$.
We define the {\em weight} of a partition $\la$
to be the weight $\alpha_{i_1}+\cdots+\alpha_{i_d} \in Q_+$
where $(i_1,\dots,i_d)$ is the residue sequence of
any standard $\la$-tableau.
The partition $\la$ 
can be uniquely recovered from its weight. 

Now fix a partition $\la$ of $d$ of weight $\al$. 
We use Khovanov-Lauda generators to construct a module 
$S(\la)$ over the block $H_\al^{\La}$ of the symmetric group 
$FS_d$. 
As a vector space, we let
\begin{equation}
S(\la):=
\bigoplus_{\T\in\Tab(\la)}F v_{\T}.
\end{equation}
The action of the Khovanov-Lauda generators on this basis
is defined as follows: 
\begin{align}\label{f1}
e(\bi)v_{\T}&:=
\left\{
\begin{array}{ll}
v_{\T} \hspace{3.5mm}&\hbox{if $\bi^{\T} = \bi$},\\
0 &\hbox{otherwise};
\end{array}
\right.
\\ 
y_r v_{\T}&:=0;\label{secd}
\\  
\psi_rv_{\T}&:=
\left\{
\begin{array}{ll}
v_{s_r\T} &\hbox{if $s_r\T\in \Tab(\la)$},\\
0 &\hbox{otherwise.}
\end{array}
\right.\label{f3}
\end{align}
It is now very easy to check that the relations (\ref{ERCyc})--(\ref{R7}) 
are satisfied, hence applying our Main Theorem in this very special case
we get an $H_\al^{\La}$-action on $S(\la)$. 
Moreover, 
if $\Stab, \T\in \Tab(\la)$ then one can obtain $\T$ from $\Stab$ by a series of basic transpositions 
so that on each step we still have a standard tableau;
see for example \cite[Lemma~2.2.8]{Kbook}. It 
follows that the $H_\al^{\La}$-module $S(\la)$ is irreducible. 

Since $H_\al^{\La}$ is a simple matrix algebra, 
our construction yields an isomorphism 
$H_\al^{\La}\cong \End_F(S(\la))$.
Letting $\la$ vary over all partitions of $d$,
we get a full set of pairwise inequivalent
irreducible $H_\al^\La$-modules in this way.
Finally, if $q=1$ and we use (\ref{ET2}), we obtain formulas for the 
action of the standard generators $s_1,\dots,s_{d-1}$ on each $S(\la)$ 
which are exactly Young's formulas. 
On the other hand if $q \neq 1$ and we use (\ref{QET2})
we obtain formulas for the action of $T_1,\dots,T_{d-1}$
on $S(\la)$ which are essentially the formulae going back to
Hoefsmit \cite{H}; see also \cite{Ram} and \cite{W}.

Replacing the set $\mathscr T(\la)$
with the set $\mathscr T_e(\la)$ of
{\em $e$-standard tableaux},
the construction of the irreducible module $S(\la)$
by the formulae (\ref{f1})--(\ref{f3})
can be generalized to include the so-called 
completely splittable irreducible
representations of the symmetric group in 
characteristic $e > 0$ (and their $q$-analogues at roots of unity);
see \cite{KCS} and \cite{R}.

\section{Base change}

Finally we explain an application of our main result to base change.
Fix $\La \in P_+$ of level $\ell$
and $\alpha \in Q_+$ of height $d$.
Henceforth we will denote the algebra $H^\La_\alpha$
(resp.\ $H^\La_d$)
instead by $H^\La_\alpha(F)$ (resp.\ $H^\La_d(F)$), 
as we are going to allow the ground field to change. 
Also make a choice of 
Khovanov-Lauda generators for $H^\La_\alpha(F)$
according to (\ref{klg}) if $q=1$ or (\ref{Qklg}) if $q \neq 1$.
To start with we explain how to descend
from the field $F$ to its prime subfield $E$.

\begin{Theorem}\label{EF}
Let $E$ be the prime subfield of $F$.
Let $H^\La_\alpha(E)$ denote the $E$-subalgebra of 
$H^\La_\alpha(F)$ generated by the Khovanov-Lauda generators.
Then the natural map 
$$
F \otimes_{E} H^\La_\alpha(E)
\rightarrow H^\La_\alpha(F)
$$ 
is an $F$-algebra isomorphism.
Moreover, letting 
$R^\La_\alpha(E)$ be the cyclotomic Khovanov-Lauda algebra
over the field $E$ defined as in $\S$\ref{sckl},
there is an $E$-algebra isomorphism
$R^\La_\alpha(E) \stackrel{\sim}{\rightarrow}
H^\La_\alpha(E)$ sending
the named 
generators of $R^\La_\alpha(E)$ to 
the Khovanov-Lauda generators of $H^\La_\alpha(E)$.
\end{Theorem}

\begin{proof}
There is an obvious surjective homomorphism 
$R^\La_\alpha(E) \twoheadrightarrow H^\La_\alpha(E)$
mapping the named generators of $R^\La_\alpha(E)$
to the Khovanov-Lauda generators of $H^\La_\alpha(E)$.
Extending scalars, this yields a map
$F \otimes_{E} R^\La_\alpha(E) \twoheadrightarrow H^\La_\alpha(F)$.
On the other hand by the presentation for $R^\La_\alpha(F)$
arising from our Main Theorem, there is a map
$H^\La_\alpha(F) \rightarrow F \otimes_{E} R^\La_\alpha(E)$
such that $e(\bi)\mapsto 1 \otimes e(\bi)$,
$y_r\mapsto 1 \otimes y_r$ and $\psi_r \mapsto 1 \otimes \psi_r$
for each 
$\bi$ and $r$. Clearly these two maps are 
mutual inverses, hence both are isomorphisms.
This implies that the original map
$R^\La_\alpha(E) \twoheadrightarrow H^\La_\alpha(E)$
is injective, so it is an isomorphism,
and the theorem follows.
\end{proof}

\begin{Corollary}\label{mcor}
Let $D(F)$ be an irreducible $H^\La_\alpha(F)$-module.
Then there exists an irreducible
$H^\La_\alpha(E)$-module $D(E)$
such that 
$D(F) \cong F \otimes_{E} D(E)$.
\end{Corollary}

\begin{proof}
By \cite[Corollary 3.19]{KL1} (extended to include the case $e=2$ as well)
every irreducible $R^\La_\alpha(E)$-module
is absolutely irreducible.
\end{proof}

Recall the
{\em formal character} $\ch M$ of a finite dimensional 
$H^\La_\alpha(F)$-module $M$
means the formal linear combination
$\sum_{\bi \in I^\alpha} (\dim M_\bi) \cdot \bi$,
where $M_\bi$ is the $\bi$-weight space $e(\bi) M\subseteq M$.
It is known that 
two modules $M$ and $N$ are equal in the Grothendieck group
if and only if $\ch M = \ch N$.
Corollary~\ref{mcor} implies at once that the formal
characters of the irreducible $H^\La_\alpha(F)$-modules
are the same as the formal
characters of the irreducible $H^\La_\alpha(E)$-modules
(defined in an analogous way via the idempotents $e(\bi) \in H^\La_\alpha(E)$).
Since the formal characters of the
Specht modules of \cite{DJM} 
do not depend on the underlying field, this proves a conjecture of Mathas
asserting that decomposition matrices of Specht modules over
$H^\La_\alpha(F)$ depend only on $e$
and the characteristic of $F$ (not on $F$ itself):

\begin{Corollary}
Conjecture 6.38 of \cite{MathasB} holds.
\end{Corollary}

Now suppose also that $K$ is a field of characteristic zero, and 
let $\xi \in K^\times$ be a primitive $e$th root of unity
if $e > 0$, or some element that is not a root of unity if $e = 0$.
Let $H^\La_d(K)$ denote the 
cyclotomic Hecke algebra over $K$ at parameter $\xi$,
and $H^\La_\alpha(K) := e_\alpha H^\La_d(K)$.
Fix a choice of Khovanov-Lauda generators
for $H^\La_\alpha(K)$ according to (\ref{Qklg}).
The following theorem explains how $H^\La_\alpha(F)$ 
can be obtained from
$H^\La_\alpha(K)$ by base change.

\begin{Theorem}\label{Red}
Let $H^\La_\alpha(\Z)$
denote the 
subring of $H^\La_\alpha(K)$ generated by
the Khovanov-Lauda generators.
Then $H^\La_\alpha(\Z)$ is a free $\Z$-module
and there are isomorphisms
\begin{align}\label{i1}
H^\La_\alpha(K)&\stackrel{\sim}{\rightarrow} 
K \otimes_\Z H^\La_\alpha(\Z),\\
H^\La_\alpha(F)
&\stackrel{\sim}{\rightarrow} F \otimes_\Z H^\La_\alpha(\Z),\label{i2}
\end{align}
such that $e(\bi) \mapsto 1 \otimes e(\bi)$,
$y_r\mapsto 1 \otimes y_r$ and $\psi_r\mapsto 1 \otimes \psi_r$
for each $\bi$ and $r$.
\end{Theorem}

\begin{proof}
Let $H^\La_\alpha(\Q)$ denote the $\Q$-subalgebra of $H^\La_\alpha(K)$
generated by the Khovanov-Lauda generators.
By Theorem~\ref{EF}, we can identify $H^\La_\alpha(K) = K \otimes_\Q 
H^\La_\alpha(\Q)$, and $H^\La_\alpha(\Z)$ can be viewed equivalently as
the subring of $H^\La_\alpha(\Q)$ generated by its Khovanov-Lauda
generators.
By the same argument as in the proof of Corollary~\ref{fda},
$H^\La_\alpha(\Z)$ is spanned as a $\Z$-module
by elements defined in terms of its Khovanov-Lauda generators
like in (\ref{sset}), and only finitely many of these
elements are non-zero. Hence $H^\La_\alpha(\Z)$ is a finitely
generated $\Z$-submodule of $H^\La_\alpha(\Q)$ 
and it generates $H^\La_\alpha(\Q)$ over $\Q$.
As $\Z$ is a principal ideal domain, this implies that
$H^\La_\alpha(\Z)$ is a lattice in $H^\La_\alpha(\Q)$,
hence it is also a lattice in $H^\La_\alpha(K)$
i.e. it is the $\Z$-span of a $K$-basis of $H^\La_\alpha(K)$.
This shows that $H^\La_\alpha(\Z)$ is a free $\Z$-module
and the canonical map $K \otimes_{\Z} H^\La_\alpha(\Z)
\rightarrow H^\La_\alpha(K)$ is an isomorphism.

Now consider the $F$-algebra $F \otimes_{\Z} H^\La_\alpha(\Z)$.
It is generated by the elements
$1 \otimes e(\bi)$,
$1 \otimes y_r$ and $1 \otimes \psi_r$
for all $\bi$ and $r$, and according to the Main Theorem
these elements satisfy 
the same relations as the defining relations of the 
Khovanov-Lauda generators of $H^\La_\alpha(F)$.
Hence there is a surjection
$$
H^\La_\alpha(F) \twoheadrightarrow
F \otimes_{\Z} H^\La_\alpha(\Z),
\quad
e(\bi) \mapsto 1 \otimes e(\bi),
y_r \mapsto 1 \otimes y_r,
\psi_r \mapsto 1 \otimes \psi_r.
$$
To complete the proof of the theorem,
we need to show that this surjection is an isomorphism. 
Note that
$\dim_F F \otimes_\Z H^\La_\alpha(\Z)
= \dim_K H^\La_\alpha(K)$ by the previous paragraph.
In view of the above surjection,
we know that $\dim_F H^\La_\alpha(F) \geq \dim_K H^\La_\alpha(K)$
for all $\alpha$, and are done if we can 
show that equality holds everywhere.
This follows because
$$
\sum_{\substack{\alpha\in Q_+\\
\height(\alpha)=d}}\!\! \dim_F H^\La_\alpha(F)
=
\dim_F H^\La_d(F) =
\dim_K H^\La_d(K) = \!\!\sum_{\substack{\alpha \in Q_+\\\height(\alpha)=d}} \!\!\dim_K H^\La_\alpha(K)
$$
by (\ref{degdim}) and (\ref{qdim}).
\end{proof}

As a consequence
we can explain how to
reduce irreducible $H^\La_\alpha(K)$-modules
modulo $p$ to obtain well-defined $H^\La_\alpha(F)$-modules,
in the spirit of the Brauer-Nesbitt construction in finite group theory.

\begin{Theorem}\label{bct} 
If $D(K)$ is an irreducible $H^\La_\alpha(K)$-module
and $0 \neq v \in D(K)$,
then $D(\Z) := H^\La_\alpha(\Z) v$ is
a lattice in $D(K)$
that is invariant under the action of $H^\La_\alpha(\Z)$.
For any such lattice,
$$
D(F) := F \otimes_\Z D(\Z)
$$ 
is naturally an $H^\La_\alpha(F)$-module
with the same formal character as $D(K)$.
In particular,  the image of $D(F)$
in the Grothendieck group is independent of the choice of lattice.
\end{Theorem}

\begin{proof}
Let $H^\La_\alpha(\Q)$ denote the $\Q$-subalgebra
of $H^\La_\alpha(K)$ generated by the Khovanov-Lauda generators.
By Corollary~\ref{mcor}, there exists 
an irreducible $H^\La_\alpha(\Q)$-submodule $D(\Q)$ of $D(K)$
such that $D(K) = K \otimes_\Q D(\Q)$. 
We can choose this so that the given non-zero vector $v$ belongs to 
$D(\Q)$.
Then $H^\La_\alpha(\Z) v$ is a finitely generated $\Z$-submodule
of $D(\Q)$ which must span $D(\Q)$ over $\Q$ since $D(\Q)$ is irreducible.
This implies that $D(\Z)$ is a lattice in $D(\Q)$, hence also
it is a lattice in $D(K)$.
The rest of the theorem follows easily since 
the $\bi$-weight spaces of $D(K)$ and $D(F)$
are equal to $e(\bi) D(K)$ and $e(\bi) D(F)$,
respectively, and
$e(\bi) D(\Z)$ is a lattice in $e(\bi) D(K)$.
\end{proof}

Theorem~\ref{bct} implies that 
there is a well-defined notion of composition multiplicity
in the reduction modulo $p$ of an irreducible $H^\La_d(K)$-module.
Since the irreducible $H^\La_d(K)$-modules
are understood by \cite{Ariki},
it would be particularly interesting to find
hypotheses on $\alpha$ that ensure that
$D(F)$ is an irreducible $H^\La_\alpha(F)$-module
for every irreducible $H^\La_\alpha(K)$-module $D(K)$.
In level one, there is a precise conjecture for this known as the
James conjecture; see e.g. \cite[$\S$2]{Geck} and \cite{F}.

\end{document}